\documentclass{amsart}
\usepackage{graphicx}
\usepackage{amssymb}
\usepackage[all]{xy}
\usepackage{hyperref}
\usepackage{subcaption}
\usepackage{amsaddr}

\newcommand{\Sone}{\mathcal{S}_1}
\newcommand{\Limb}{\mathcal{L}}
\newcommand{\R}{\mathbb{R}}
\newcommand{\Z}{\mathbb{Z}}
\newcommand{\C}{\mathbb{C}}
\newcommand{\CC}{\widehat{\C} }
\newcommand{\D}{\mathbb{D} }
\newcommand{\Q}{\mathbb{Q}}
\newcommand{\T}{\mathbb{T}}
\newcommand{\Hy}{\mathcal{H}}
\newcommand{\RL}{\mathfrak{R}}
\newcommand{\pray}{\mathcal{R}}
 \newcommand{\mate}{\perp \! \! \! \perp}
 \newcommand{\Bot}{\text{B\"{o}t}}
 \renewcommand{\Re}{\mathrm{Re}}
 \renewcommand{\Im}{\mathrm{Im}}
 
 \newcommand{\OC}{\text{\textcopyright}}

\newtheorem{thm}{Theorem}[section]
 \newtheorem{cor}[thm]{Corollary}
 \newtheorem{lem}[thm]{Lemma}
 \newtheorem{prop}[thm]{Proposition}
 \newtheorem{conj}[thm]{Conjecture}
 \theoremstyle{definition}
 \newtheorem{defn}[thm]{Definition}
 \theoremstyle{remark}
 \newtheorem{rem}[thm]{Remark}
 \newtheorem{ex}[thm]{Example}

\title[Matings in $\Sone$ and $\alpha$-symmetry]{Matings of cubic polynomials with a fixed critical point. Part II: $\alpha$-symmetry of limbs}
\author{Thomas Sharland}
\address{Department of Mathematics and Applied Mathematical Sciences, University of Rhode Island, RI 02881, U.S.A}
\email{tsharland@uri.edu}
\date{} 
\begin{document}

\begin{abstract}
Following Milnor, define $\Sone$ to be the space of monic, centred cubic polynomials with a marked fixed critical point. In this article, a sequel to \cite{CubicObs}, we provide a combinatorial sufficient (and conjecturally, necessary) condition (called $\alpha$-\emph{symmetry}) for the mating of two postcritically finite polynomials in $\Sone$ to be obstructed.  To do this, we study the rotation sets associated to the parameter limbs in the connectedness locus of $\Sone$, which allows us to determine when there exist ray classes in the formal mating which contain a closed loop.  We give a proof of the necessity of $\alpha$-symmetry for a particular subset of postcritically finite maps in $\Sone$. Many examples are given to illustrate the results of the paper.
\end{abstract}
 
 
\maketitle

\section{Introduction}

In this article we discuss the mating of two cubic polynomials with a fixed critical point. Informally, the mating construction allows us to take two complex polynomials $f$ and $g$ (along with their filled Julia sets $K(f)$ and $K(g)$) and paste them together to construct branched covering of the sphere. If this branched covering (or a suitable modification of it) is (Thurston) equivalent to a rational map, we say that $f$ and $g$ are mateable.

In the quadratic case, it was shown through work of Rees, Shishikura and Tan \cite{Tan:quadmat} that two postcritically finite polynomials with locally connected Julia sets were mateable if and only if they do not belong to conjugate limbs of the Mandelbrot set. In this case, the obstructions obtained are topological; the space obtained from pasting together the filled Julia sets fails to be a sphere. In other words, two such polynomials are mateable if and only if they are topologically mateable.

However, Shishikura and Tan \cite{ShishTan} showed that there exist cubic polynomials which are topologically mateable, but the resulting branched covering fails to be equivalent to a rational map. This indicates that the degree $3$ (and higher) case is more complicated. Restricting attention to $\Sone$, the space of monic, centred cubic polynomials with a marked fixed critical point, the present author showed in \cite{CubicObs} that any obstruction to the mating of postcritically finite polynomials in $\Sone$ is topological, paralleling the quadratic case. However, in contrast to the quadratic case, we will prove the following.

\begin{thm}\label{introthm}
 There exist postcritically finite polynomials $F_a$, $F_b$ in $\Sone$ which do not lie in conjugate limbs of $\Sone$ but for which $F_a$ and $F_b$ are not mateable.
\end{thm}

Thus the necessity of the conjugate limb condition does not carry over from quadratic polynomials to the space $\Sone$. We provide a sufficient (combinatorial) condition for two postcritically finite polynomials in $\Sone$ to fail to be mateable (Proposition~\ref{p:limbcondition}). After this, we show that this condition can be satisfied even when the two polynomials do not belong to conjugate limbs of $\Sone$ (Theorem~\ref{MainThm}). We  conjecture that the given sufficient condition is also necessary, and give a proof of necessity in a special case (Theorem~\ref{t:preperneccessity}). The paper also contains a discussion of the combinatorics of limbs in the space $\Sone$, which is used to provide the combinatorial description mentioned above. Examples are given to illuminate the main results.

\subsection*{Acknowledgments} The author wishes to thank the anonymous referee, whose suggestions helped improve the clarity and exposition of the paper. Many of the figures of Julia sets of rational maps in this paper were created using \emph{Dynamics Explorer} \cite{DynEx}.

\section{Preliminaries}

We presume some familiarity with the dynamics of rational functions on the Riemann sphere: for background, the reader is referred to \cite{MilnorComplex}. 

\subsection{Rotation sets}

In this section we discuss the results we need on rotation sets. Much of the background material is taken from \cite{SaeedBook} (see also \cite{RotSubsets}). Let $\T = \R / \Z$ be the circle. Any continuous $g \colon \T \to \T$ lifts, via the canonical projection, to a continuous $G \colon \R \to \R$ which satisfies $G(x+1) = G(x) + d$ for some $d \in \Z$ called the degree of $g$. Furthermore, we say $g$ is \emph{monotone} if $G$ is monotone.

Given a continuous non-decreasing map $G \colon \R \to \R$ such that $G(x+1) = G(x) + 1$ for all $x \in \R$, the limit
\[
 \tau(G) = \lim_{n \to \infty} \frac{G^{\circ n}(x) - x}{n}
\]
is independent of $x$, and is called the \emph{translation number} of $G$. The \emph{rotation number} $\rho(g)$ of a degree $1$ monotone map $g \colon \T \to \T$ is then defined as the residue class of $\tau(G)$ modulo $\Z$, where $G$ is any lift of $g$.

For $d \geq 2$, we define $m_d \colon \T \to \T$ by
\[
 m_d(t) = dt \pmod \Z.
\]
This paper will focus on $m_2$ and $m_3$, since both play an important role in the dynamics of maps in $\Sone$. A non-empty compact set $X \subseteq \T$ is called a \emph{rotation set} if $m_d(X) = X$ and the restriction $m_d|_X$ can be extended to a degree $1$ monotone map on $\T$. Since any two such extensions $g$ and $h$ agree on $X$, it follows from Theorem 1.8 of \cite{SaeedBook} that $\rho(g) = \rho(h)$. We denote this common rotation number by $\rho(X)$ (or just $\rho$ if the set $X$ is clear in the context) and call it the \emph{rotation number of} $X$.

Given a rotation set $X$ for $m_d$, we call a component of $\T \setminus X$ a \emph{gap} of $X$. A gap is called a \emph{major} gap if its length $\ell$ satisfies $\ell \geq 1/d$. The \emph{multiplicity} of a major gap is the integer part of $d \ell$. If a gap is not major, then it is called a \emph{minor} gap. A minor gap maps homeomorphically onto its image (which must also be a gap of $X$) under $m_d$. We collect some important consequences of the above.

\begin{prop}[Theorem 2.9, \cite{SaeedBook}]\label{prop:gapthm}
 A rotation set is uniquely determined by its major gaps.
\end{prop}

\begin{prop}[Corollary 2.16, \cite{SaeedBook}]\label{prop:rotsets}
 Let $X \subset \T$ be a non-empty $m_d$-invariant set. Then $X$ is a rotation set if and only if $\T \setminus X$ contains $d-1$ disjoint intervals, each of length $1/d$.
\end{prop}

We will refer to a periodic orbit under $m_d$ as a \emph{cycle}. When $d=3$, a rotation set can contain at most two cycles $C$ and $C'$, each with the same rotation number $\rho$ \cite{SaeedBook}. Furthermore, these cycles are \emph{superlinked} in the sense that each gap of $C$ meets $C'$, and each gap of $C'$ meets $C$. If $X$ is rotation set under $m_3$ with $\rho(X) \in \Q$ made up of the union of two cycles, we will define the \emph{signature} $s(X)$ to be vector $(s_1,s_2)$, where $s_1$ is the number of elements of $X$ in $[0,1/2)$ and $s_2$ is the number of elements of $X$. The following, a special case of a theorem of Goldberg, gives a characterisation of which signatures can occur when $X$ is the union of two cycles under $m_3$.

\begin{prop}[\cite{Goldberg}]\label{p:signaturetheorem}
Suppose $0 \leq p/q < 1$ is a reduced fraction and $0 \leq s_1 \leq s_2 = 2q$. Then there exists a rotation set under $m_3$ with rotation number $p/q$ and signature $s(X) = (s_1,s_2)$ made up of the union of two cycles if and only if $s_1$ is odd. Furthermore, if this rotation set exists, then it is unique. 
\end{prop}

We remark that for $\rho = 0$, the unique $m_3$-rotation set made up of the union of two cycles is $X = \{ 0, 1/2 \}$, since $0$ and $1/2$ are the two fixed points of $m_3$. 

%%%%%%%%%%%%%%%%%%%%%%%%%%%%%%%%%%%%%%%%%%%%%%%%%%%%%%%%%%%%%%

\subsection{Cubic Polynomials with a fixed critical point}

 We now discuss some preliminaries on the dynamics of cubic polynomials. Much of the material we need is from \cite{CP1}. Many properties of the dynamical and parameter spaces of polynomials in $\Sone$ were also investigated in \cite{Roesch} and \cite{FaughtThesis}. Denote by $\Sone$ the space of monic, centred cubic polynomials with a marked fixed critical point. Under a holomorphic change of coordinates, any such polynomial is of the form
\[
 F_a(z) = z^3 - 3a^2z+2a^3 + a.
\]
Under this parameterisation $a$ is the (fixed) \emph{marked critical point} and $-a$ is the \emph{free critical point} of $F_a$. We will call $v_a = F_a(-a)$ the \emph{free critical value} of $F_a$. The normal form above provides a conformal isomorphism $\Sone \cong \C$ via $F_a \mapsto a$. The \emph{connectedness locus} $\mathcal{C}(\Sone)$ is the set of polynomials $F_a$ for which the Julia set is connected; equivalently, it is the set of polynomials $F_a$ for which the free critical point $-a$ has bounded orbit. For any map $F_a \in \Sone$, we denote by $U_a$ the Fatou component containing the fixed critical point $a$ (thus $U_a$ is the immediate basin of $a$). The \emph{main hyperbolic component} $\Hy_0$ is the hyperbolic component in $\Sone$ which contains $F_0$, the function $z \mapsto z^3$ which has a double critical fixed point.

\subsubsection{Structure of filled Julia sets} For all $F_a \in \Hy_0$, the Julia set of $F_a$ is the simple closed curve $\partial U_a$ and $-a \in U_a$. Now suppose $F = F_a \in \mathcal{C}(\Sone) \setminus \Hy_0$. Then, by B\"ottcher's theorem, there exists a conformal isomorphism $ \Bot = \Bot_a \colon U_a \to \mathbb{D}$ which satisfies
\[   
 \Bot(F(z)) = (\Bot(z))^2.
\]
By a theorem of Roesch \cite{RoeschPuzzles} (see also \cite{FaughtThesis,Roesch}), $\partial U_a$ is a simple closed curve. Hence by Carath\'eodory's Theorem, there exists a homeomorphic extension of $\Bot$ to the boundary $\partial U_a$. By abuse of notation, we will also use $\Bot$ for this homeomorphism from $\overline{U}_a$ to $\overline{\mathbb{D}}$. Thus for each point $z \in \partial U_a$, there exists a unique $t \in \T$ such that $z = \Bot(e^{2 \pi i t})$. For brevity we will write $\beta(t)$ instead of $\Bot(e^{2 \pi i t})$ and call $t$ the internal angle of $z$. The following theorem then describes the structure of the (filled) Julia set of a map in $\mathcal{C}(\Sone) \setminus \Hy_0$.

\begin{thm}[Theorem 2.1, \cite{CP1}]\label{t:dynlimbs}
 If $F_a \in \mathcal{C}(\Sone)$ with $-a \notin U_a$, then the filled Julia set $K(F_a)$ is equal to the union of the topological disk $\overline{U}_a$ with a collection of compact connected sets $K_t$, where $t$ ranges over a countable subset $\Lambda \subseteq \T$. Furthermore
 \begin{enumerate}
  \item The $K_t$ are pairwise disjoint, and each $K_t$ intersects $\overline{U}_a$ in the single boundary point $\beta(t)$.
  \item There is a distinguished element $t_0 \in \Lambda$ such that the free critical point $-a$ belongs to $K_{t_0}$.
  \item If $t \in \T$, then $t \in \Lambda$ if and only if there exists some $n \geq 0$ such that $2^nt \equiv t_0 \pmod \Z$.
  \item For $t \not\equiv t_0 \pmod \Z$, the map $F_a$ carries $K_t$ homeomorphically onto $K_{2t}$. However, $F_a$ maps $K_{t_0}$ onto the entire filled Julia set $K(F_a)$.
 \end{enumerate}
\end{thm}

We call $K_t$ a \emph{(dynamical) limb} for the map $F$; the point $\beta(t)$ where $K_t$ meets $U_a$ is called the \emph{root point} of the limb $K_t$. The set $K_{t_0}$ is called the \emph{critical limb} of $F_a$. 

Suppose a polynomial $F$ has a connected filled Julia set $K(F)$. By B\"{o}ttcher's theorem, there exists a conformal isomorphism  $\Bot_F \colon \CC \setminus \overline{\D} \to \CC \setminus K(F)$ which conjugates the map $z \mapsto z^d$ on $\CC \setminus \overline{\D}$ to the map $F$ on $\CC \setminus K(F)$. Furthermore, there is a unique choice of $\Bot_F$ which is tangent to the identity at infinity.

\begin{defn}\label{d:extray}
 The \emph{external ray} of angle $t$ for the polynomial $F$ is 
\[
  R_{F}(t) = \Bot_F (r_t)
\]
where $r_t = \{ r \exp(2 \pi i t) : r > 1 \} \subset \C \setminus \D$ is the corresponding radial line.
\end{defn}

We will often use the notation $R_a(t)$ instead of $R_{F_a}(t)$ for a polynomial $F_a \in \Sone$. If the choice of polynomial is clear, we sometimes drop the subscript altogether. By Carath\'{e}odory's theorem, the filled Julia set $K(F)$ (or equivalently, the Julia set $J(F)$) of $F$ is locally connected if and only if $\Bot_F$ extends continuously to the unit circle. The induced semiconjugacy $\gamma_F \colon \R / \Z \to J(F)$ is called the \emph{Carath\'{e}odory loop}. The point $\gamma_F(t) \in J(F)$ is called the \emph{landing point} of the ray $R_F(t)$ and a point $z$ of the Julia set is called \emph{biaccessible} if it is the landing point of (at least) two external rays. That is, there exists $s \neq t$ such that $z = \gamma_F(s) = \gamma_F(t)$.

\subsubsection{Parameter Limbs}\label{ss:paramlimbs}

We will also need the notion of limbs in the parameter space. For a map $F_a \in \Hy_0$, then similarly to the above, there exists a unique conformal isomorphism $\Bot$ on a neighbourhood of $a$ conjugating $F_a$ to the map $w \mapsto w^2$. By Lemma 3.6 in \cite{CP1}, $\Bot$ can be analytically continued to a well-defined holomorphic function on a neighbourhood of the free critical point $-a$. We may then define a  holomorphic function $\phi \colon \Hy_0 \to \mathbb{D}$ by $\phi(a) = \Bot_a(-a)$. This map $\phi$ then extends to a continuous map $\phi \colon \mathcal{C}(\Sone) \to \overline{\mathbb{D}}$ such that if $F \in \mathcal{C}(\Sone) \setminus \Hy_0$ then $\phi(F) = e^{2 \pi i t_0}$, where $t_0 \in [0,1)$ is the internal argument of the critical limb of $F$.

Let $t_0 \in (0,1)$. We define the \emph{(parameter) limb} $\Limb_{t_0}^{+}$ to be the set
\[
 \Limb_{t_0}^{+} = \{ F_a \in \mathcal{C}(\Sone) \mid \phi(F) = e^{2 \pi i t_0} \text{ and } \Re(a) > 0 \}
\]
and analogously
\[
 \Limb_{t_0}^{-} = \{ F_a \in \mathcal{C}(\Sone) \mid \phi(F) = e^{2 \pi i t_0} \text{ and } \Re(a) < 0 \}.
\] 
We will show that one can distinguish $\Limb_{t_0}^+$ from $\Limb_{t_0}^-$ by dynamical means in Lemma~\ref{l:fixedraylanding}. For $t_0=0$ we define $\Limb_0^+ = \{  F_a \in \mathcal{C}(\Sone) \mid \phi(F) = e^{2 \pi i 0} \text{ and } \Im(a) < 0 \}$ and $\Limb_0^- = \{  F_a \in \mathcal{C}(\Sone) \mid \phi(F) = e^{2 \pi i 0} \text{ and } \Im(a) > 0 \}$. Note that by the $180^{\circ}$ rotational symmetry of $\Sone$, we have $F_a \in \Limb_{t_0}^+$ if and only if $F_{-a} \in \Limb_{t_0}^-$. The root point of a limb $\Limb_{t_0}^\pm$ is defined to be the unique point of $\Limb_{t_0}^\pm \cap \partial \Hy_0$. 

In much of this paper, we will have to deal with the special case $t_0 = 0$ separately from other periodic internal arguments. Fortunately, most of the time the case for $t_0=0$ can be proved by inspection. When such a consideration is made, it will be clearly indicated.

We will also need the notion of parameter rays in the complement of the connectedness locus. They are closely related to the concept of external rays, as the following definition shows.

\begin{defn}
We say a map $F_a \in \Sone$ belongs to the \emph{parameter ray} $\pray(\theta)$ if the external ray $R_a(\theta)$ passes through the cocritical point $2a$ in the dynamical plane of $F_a$.
\end{defn}

If $\theta$ is $m_3$-periodic, then the rays $\pray(\theta-1/3)$ and $\pray(\theta+1/3)$ land on a parabolic map. In this case the angles $\theta-1/3$ and $\theta+1/3$ are called \emph{co-periodic}. The following properties of parameter limbs come from \cite{CP1} (Compare also \cite{CP3} for a more in-depth discussion). See Figure~\ref{f:S1limbs} for some examples of parameter limbs.

\begin{figure}
\centering
\includegraphics[width=0.5\linewidth]{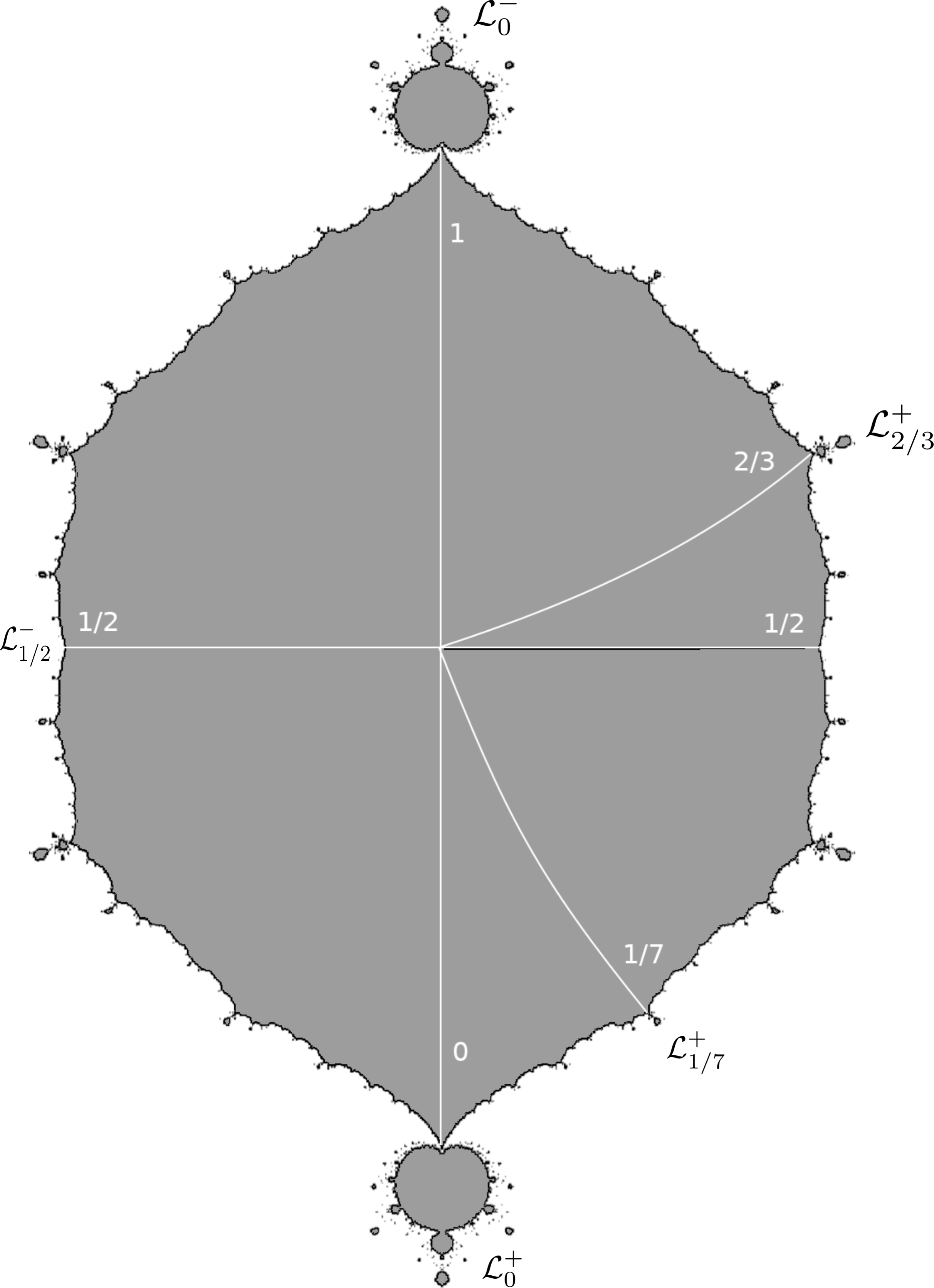}
\caption{The connectedness locus of $\Sone$. Internal rays of the main hyperbolic component $\Hy_0$ are shown, as well as some associated limbs.}
\label{f:S1limbs}
\end{figure}

\begin{prop}\label{prop:limblanding}
 Suppose $t_0$ is periodic of period $q$ under $m_2$. Then there are two external parameter rays $\pray(\theta)$ and $\pray(\theta')$, where $\theta =\frac{i}{3(3^q-1)}$ and $\theta' = \frac{i+1}{3(3^q-1)}$, which land at the root point of $\Limb = \Limb_{t_0}^{\pm}$. Furthermore, for any map $F_a \in \Limb$, the internal ray of angle $t_0 + \frac{1}{2}$ in $U_a$, as well as the external rays $R_a(\theta)$ and $R_a(\theta')$ land at a common pre-periodic point in $J(F_a)$.
\end{prop}

The above result implies the following.

\begin{cor}\label{cor:critlimbrays}
Let $\theta = \frac{i}{3(3^q-1)}$ and $\theta' = \frac{i+1}{3(3^q-1)}$. If the two parameter angles $\pray(\theta)$ and $\pray(\theta')$ land at the root point of $\Limb = \Limb_{t_0}^{\pm}$, then $R_a(\theta +1/3)$ and $R_a(\theta' - 1/3)$ land at the root of the critical limb of any map $F_a \in \Limb$.
\end{cor}

\begin{proof}
Let $F_a \in \Limb$. The landing point $z$ of the internal ray of angle $t_0 + \frac{1}{2}$ in $U_a$ has the same image as the landing point $r$ of the internal ray of angle $t_0$ in $U_a$. By definition, $r$ is the root point of the critical limb of $F_a$. Since by Proposition~\ref{prop:limblanding} $R_a(\theta)$ and $R_a(\theta')$ land at $z$, then $R_a(3\theta)$ and $R_a(3\theta')$ land on $F_a(z) = F_a(r)$. Thus one of the rays $R_a(\theta + 1/3)$ and $R_a(\theta-1/3)$ must land on $r$, and similarly one of the rays $R_a(\theta' + 1/3)$ and $R_a(\theta'-1/3)$ must land on $r$. But since the critical limb maps onto the whole filled Julia set, the difference between the angles of the rays landing on $r$ must be greater than $1/3$. But the only choice of angles from the above options which achieve this are $\theta+1/3$ and $\theta'-1/3$. Thus $R_a(\theta +1/3)$ and $R_a(\theta' - 1/3)$ land on $r$.
\end{proof}

In the case where $t_0$ is strictly preperiodic under $m_2$, the limb $\Limb_{t_0}^\pm$ consists of a single function $F_a \in \partial \Hy_0$, which is the landing point of exactly one parameter ray. Furthermore, it follows from Theorem \ref{t:dynlimbs} that for this map $F_a$, the root $\beta(t_0)$ of the critical limb is the free critical point, and so the biaccessible points on $\partial U_a$ are $\beta(t_0)$ and its preimages. In particular, $F_a$ has no biaccessible periodic points.  

% Indeed, the Hubbard tree (the convex hull in the Julia set of the critical orbits) for $F_a$ is contained in the closure of $U_a$, and since all biaccessible periodic points belong to the Hubbard tree, we see that $F_a$ has no biaccessible periodic points in its Julia set.

%%%%%%%%%%%%%%%%%%%%%%%%%%%%%%%%%%%%%%%%%%%%%%%%%%%%

\subsection{Matings of postcritically finite polynomials}

In this section we outline the mating construction. Much of the material here can be found in e.g \cite{ShishTan,Tan:quadmat}.

\subsubsection{Formal Mating} Let $f$ and $g$ be monic degree $d$ polynomials. We define
\[
 \OC = \C \cup \{ \infty \cdot e^{2 \pi i t} : t \in \R / \Z \},
\]
the complex plane compactified with the circle of directions at infinity. We then continuously extend the two polynomials to the circle at infinity by defining
\[
  f(\infty \cdot e^{2 \pi i t}) =  \infty \cdot e^{2 d \pi i t} \quad \text{and} \quad g(\infty \cdot e^{2 \pi i t}) =  \infty \cdot e^{2 d \pi i t}.
\]
Label this extended dynamical plane of $f$ (respectively $g$) by $\OC_f$ (respectively $\OC_g$). We create a topological $2$-sphere $\Sigma_{f,g}$ by gluing the two extended planes together along the circle at infinity:
\[
 \Sigma_{f,g} = (\OC_f \sqcup \OC_g) / \sim
\]
where $\sim$ is the relation which identifies the point $\infty \cdot e^{2 \pi i t} \in \OC_f$ with the point $\infty \cdot e^{- 2 \pi i t} \in \OC_g$. The \emph{formal mating} is then defined to be the branched covering $f \uplus g \colon \Sigma_{f,g} \to \Sigma_{f,g}$ such that
\begin{align*}
    f \uplus g|_{\OC_f} \, =& \, f \quad \textrm{and} \\
    f \uplus g|_{\OC_g} \, =& \, g.
\end{align*} 

\subsubsection{Topological Mating}

A related concept to the formal mating is the topological mating. Let $f$ and $g$ be two monic degree $d$ polynomials with locally connected Julia set. First, construct the space $\Sigma_{f,g}$ as above, and define the \emph{ray-equivalence relation} $\approx$ as follows. We denote by $\sim_f$ the smallest equivalence relation on $\OC_f$ such that $x \sim_f y$ if and only if $x,y \in \overline{R}_{f}(t)$ for some $t$; the equivalence relation $\sim_g$ is defined analogously. Now $\approx$ is the smallest equivalence relation on $\Sigma_{f,g}$ which is generated by $\sim_f$ on $\OC_f$ and $\sim_g$ on $\OC_g$. The class $[x]$ of a point $x \in \Sigma_{f,g}$ will be called a \emph{ray-equivalence class}.

Evidently, there is a surjection $K(f) \sqcup K(g) \to \Sigma_{f,g} / \approx$ given by the composition
\[
  K(f) \sqcup K(g) \hookrightarrow \OC_f \sqcup \OC_g \to \Sigma_{f,g} \to \Sigma_{f,g} / \approx. 
\]
Furthermore, since $x \approx y$ implies $(f \uplus g) (x) \approx (f \uplus g) (y)$, we can uniquely define the topological mating as the map $f \mate g$ so that
\[
 \xymatrix{ K(f) \sqcup K(g) \ar[rr]^{f \uplus g} \ar[dd] && K(f) \sqcup K(g) \ar[dd]
	    \\ \\
	    \Sigma_{f,g} / \approx \ar[rr]_{f \mate g} && \Sigma_{f,g} / \approx}
\]
commutes. Under favourable conditions, the topological space $\Sigma_{f,g} / \approx$ is a topological 2-sphere, in which case we say that $f$ and $g$ are \emph{topologically mateable}. In fact, this occurs if and only if the equivalence relation $\approx$ is of Moore-type (see \cite{NotionsofMating}). That is, $\approx$ is a non-trivial closed relation such that each equivalence class is closed and no equivalence class separates the sphere $\Sigma_{f,g}$.

\subsubsection{Thurston's Theorem}
 
 Let $F \colon \Sigma \to \Sigma$ be an orientation-preserving branched self-covering of a topological $2$-sphere. We denote by $\Omega_F$ the critical set of $F$ and define
\[
 P_F = \bigcup_{n > 0} F^{\circ n}(\Omega_F)
\]
to be the postcritical set of $F$. We say that $F$ is postcritically finite if $|P_F| < \infty$. We call $F \colon \Sigma \to \Sigma$ a \emph{Thurston map} if it is a postcritically finite orientation-preserving branched self-covering of a topological $2$-sphere.
 
 \begin{defn}\label{d:Thurst}
Let $F \colon \Sigma \to \Sigma$ and $\widehat{F} \colon \widehat{\Sigma} \to \widehat{\Sigma}$ be Thurston maps. An \emph{equivalence} is given by a pair of orientation-preserving homeomorphisms $(\Phi,\Psi)$ from $\Sigma$ to $\widehat{\Sigma}$ such that 
    \begin{itemize}
     \item{$\Phi |_{P_{F}} = \Psi |_{P_{F}}$}
     \item{The following diagram commutes:
        \[
             \xymatrix{       (\Sigma,P_F) \ar[r]^{\Psi} \ar[d]_{F}    & (\widehat{\Sigma},P_{\widehat{F}}) \ar[d]^{\widehat{F}}
            \\ 
                (\Sigma,P_F) \ar[r]_{\Phi}                       & (\widehat{\Sigma},P_{\widehat{F}}) }
        \]}
    \item{$\Phi$ and $\Psi$ are isotopic via a family of homeomorphisms $t \mapsto \Phi_{t}$ which is constant on $P_F$.}
        \end{itemize}
\end{defn}

If there exists an equivalence as above, we say that $F$ and $\widehat{F}$ are \emph{equivalent}. Note that in particular, a postcritically finite rational map $R \colon \CC \to \CC$ on the Riemann sphere is a Thurston map. Hence it is natural to ask when a general Thurston map is equivalent to a rational map. In particular, since the formal mating of two postcritically finite polynomials is itself a Thurston map, we can ask if a given formal mating is equivalent to a rational map. This requires the notion of a \emph{Thurston obstruction}, which we define below.

\begin{defn}
Let $F$ be a Thurston map. A \emph{multicurve} is a collection $\Gamma = \{ \gamma_1, \ldots,\gamma_n \}$ of simple, closed,  non-peripheral curves such that each $\gamma_i \in \Gamma$ is disjoint from each other $\gamma_j$ and the $\gamma_i$ are pairwise non-homotopic relative to $P_F$. A multicurve is called $F$-\emph{stable} if for all $\gamma_{i} \in \Gamma$, all the non-peripheral components of $F^{-1}(\gamma_{i})$ are homotopic relative to $P_F$ to elements of $\Gamma$. If $\Gamma$ is a multicurve (not necessarily $F$-stable), we define the non-negative matrix $F_\Gamma = (f_{ij})_{n \times n}$ as follows. 
\[
 f_{ij} = \sum_{\gamma' \in F^{-1}(\gamma_j), \gamma' \sim \gamma_i} \frac{1}{\deg F  \colon \gamma' \to \gamma_{j}}
\]
where $\deg$ denotes the degree of the map and $\gamma' \sim \gamma_i$ denotes that $\gamma'$ is homotopic to $\gamma_i$ rel $P_F$. By standard results on non-negative matrices (see \cite{Gantmacher}), this matrix $(f_{ij})$ will have a leading non-negative eigenvalue $\lambda$. We write $\lambda(\Gamma)$ for the leading eigenvalue associated to the multicurve $\Gamma$.
\end{defn}

% Let $\gamma_{i,j,\alpha}$ be the components  of $F^{-1}(\gamma_{j})$ which are homotopic to $\gamma_{i} \in \Gamma$ in $\Sigma \setminus P_{F}$. Now define
% \[
% 	F_{\Gamma}(\gamma_{j}) = \sum_{i,\alpha} \frac{1}{\deg F |_{\gamma_{i,j,\alpha}} \colon \gamma_{i,j,\alpha} \to \gamma_{j}} \gamma_{i}.
% \] 

The importance of the above is due to the following rigidity theorem. A proof can be found in \cite{DouadyHubbard:Thurston} or \cite{HubbardTeichvol2}.

 \begin{thm}[Thurston]\label{t:Thurston}\mbox{}
\begin{enumerate}
  \item{A Thurston map $F \colon \Sigma \to \Sigma$ of degree $d \geq 2$ with hyperbolic orbifold is equivalent to a rational map $R \colon \CC \to \CC$ if and only if there are no $F$-stable multicurves with $\lambda(\Gamma) \geq 1$.}
  \item{Any Thurston equivalence of rational maps $F$ and $\widehat{F}$ with hyperbolic orbifolds is represented by a M\"{o}bius conjugacy.}
\end{enumerate}
\end{thm}

An $F$-stable multicurve with $\lambda(\Gamma) \geq 1$ is called a \emph{(Thurston) obstruction}. The condition that $F$ has a hyperbolic orbifold is a purely combinatorial one; this condition can be checked by inspecting the dynamics on the union of the critical and postcritical sets of $F$. We remark that since the maps in this article have two fixed critical points, the respective orbifolds are guaranteed to be hyperbolic. For further details, see \cite{DouadyHubbard:Thurston}. Following \cite{ShishTan}, we will say that a Thurston obstruction $\Gamma$ is \emph{irreducible} if the associated matrix $F_\Gamma$ is irreducible. The following result from \cite{ShishTan} (see also \cite{Tan:quadmat}) shows that it suffices to search for irreducible obstructions instead of Thurston obstructions. 

\begin{prop}\label{p:irreducible} (\cite{ShishTan}, Lemma 3.5)
A Thurston map $F$ with hyperbolic orbifold is not equivalent to a rational map if and only if $F$ has an irreducible obstruction.
\end{prop}

In this paper, we will mainly be concerned with a special kind of obstruction: Levy cycles. 

% We observe that the formal mating of two postcritically finite polynomials is a Thurston map. Thus one can ask if the mating is obstructed (that is, it has a Thurston obstruction).

\begin{defn}\label{d:Levy}
 A multicurve $\Gamma = \{ \gamma_{1}, \ldots , \gamma_{n} \}$ is a \emph{Levy cycle} if for each $i =1,\ldots,n$, the curve $\gamma_{i-1}$ (or $\gamma_{n}$ if $i = 1$) is homotopic to some component $\gamma_{i}'$ of $F^{-1}(\gamma_{i})$ (rel $P_{F}$) and the map $F \colon \gamma_{i}' \to \gamma_{i}$ is a homeomorphism. Furthermore
 \begin{itemize}
 \item We say $\Gamma$ is a \emph{degenerate} Levy cycle if the connected components of $\Sigma \setminus \bigcup_{i=1}^n \gamma_i$ are $D_1, \ldots , D_n$  and an additional component $C = \Sigma \setminus \bigcup_{i=1}^n \overline{D}_i$, where each $D_i$ is a disk and moreover for each $i$ the preimage $F^{-1}(D_{i+1})$ contains a component $D_{i+1}'$ which is isotopic to $D_i$ relative to $P_F$ and is such that $F \colon D_{i+1}' \to D_{i+1}$ is a homeomorphism.
 \item A degenerate Levy cycle is called a \emph{removable} Levy cycle if in addition, for all $k \geq 1$ and all $i$, the components of $F^{-k}(D_i)$ are disks.
 \end{itemize}
\end{defn} 

Note that a Levy cycle is an example of an irreducible obstruction. There is a close relationship between Levy cycles and periodic cycles of ray cycles, as given in the following theorem of Tan.

 \begin{prop}[\cite{ShishTan,Tan:quadmat}]\label{p:limset}
  Let $F = f \uplus g$. Then every Levy cycle $\Gamma$ of $F$ corresponds to a unique cycle of ray classes $[x_0],[x_1],\dotsc,[x_{n-1}]$ (called the limit set of $\Gamma$) and finitely connected $X_i \subset [x_i]$ for each $i$ such that $F \colon [x_i] \to [x_{i+1}]$ and $F \colon X_i \to X_{i+1}$ are homeomorphisms. Moreover, $\Gamma$ is isotopic to a tubular neighbourhood of $\bigcup_{i=0}^{n-1} X_i$. In particular
  \begin{enumerate}
  \item If $\Gamma$ is degenerate, then each $X_i$ is a tree and contains at least two elements of $P_F$.
  \item If $\Gamma$ is not degenerate, each $X_i$ contains a closed loop.
  \end{enumerate}
  Conversely, if $[x]$ is a periodic ray class which contains a closed loop or at least two postcritical points, then each boundary curve of a tubular neighbourhood of $[x]$ generates a Levy cycle.
 \end{prop}

 \begin{cor}\label{cor:twocritvals}
  Let $\Gamma$ be a non-removable degenerate Levy cycle of the mating $F = f \uplus g$ of postcritically finite polynomials. Then there exists a ray class under the mating which contains two critical values of $F$.
 \end{cor}

 \begin{proof}
  First, note that by the Riemann-Hurwitz formula, if $D$ is a disk such that $F^{-1}(D)$ is not a disk, then $D$ must contain (at least) two critical values. Since $\Gamma$ is not removable, there exists a disk component $\Delta$ of $\Sigma \setminus \Gamma$ and some $k$ such that $F^{-k}(\Delta)$ contains a non-disk component $C$.
  
  By Proposition~\ref{p:limset}, there exists some $i$ such that $\partial \Delta$ is isotopic to a tubular neighbourhood of $X_i$ which is a subset of a periodic ray class $[x_i]$. Thus this tubular neighbourhood (isotopically) contains $\Delta$. We claim that there exists a preimage of $[x_i]$ which contains two critical values of $F$. If not, then no preimage of $X_i \subset [x_i]$  can contain two critical values of $F$. But by Riemann-Hurwitz, all preimages of a tubular neighbourhood of $X_i$ would have to be disks. But then all preimages of $\Delta$ would also be disks, which is a contradiction. Thus $[x_i]$ must contain two critical values of $F$.
 \end{proof}

We remark that the converse of the corollary is not true: it is possible that two critical values belong to the same ray class but the mating is not obstructed, see Example~\ref{ex:critptsidentified}. The following theorem from \cite{CubicObs} indicates that all obstructions to matings of pairs of postcritically finite polynomials in $\Sone$ are topological. 

 \begin{thm}\label{mthm2}
  Let $f$ and $g$ be monic postcritically finite polynomials in $\Sone$. Then any obstruction to the formal mating $f \uplus g$ contains a Levy cycle. 
 \end{thm}

 Theorem~\ref{mthm2} does not hold for general cubic polynomials: in \cite{ShishTan}, the authors constructed a mating of two postcritically finite cubic polynomials which was obstructed, but this obstruction was not a Levy cycle. Both the polynomials had critical points of period $3$. 
 
Theorem~\ref{mthm2} allows us to immediately prove the following.
 
\begin{cor}\label{c:postcritfinmatings}
Let $t_0$ be strictly preperiodic under $m_2$ and suppose $F_a$ is the unique element of the limb $\Limb_{t_0}^\pm$. If $F_b$ is a postcritically finite polynomial in $\Sone$ and the mating $F_a \uplus F_b$ is obstructed, then this obstruction contains a degenerate Levy cycle. 
\end{cor}

\begin{proof}
By Theorem~\ref{mthm2}, any obstruction of the mating contains a Levy cycle. By Proposition~\ref{p:limset}, if the obstruction was non-degenerate Levy cycle, then there would have to be a periodic ray cycle containing a closed loop. But $F_a$ has no biaccessible periodic points, so this is impossible.
\end{proof}

\subsubsection{Degenerate Matings}

We now discuss a technique due to Shishikura \cite{ShishEssMate} (and also \cite{ShishTan}) which allows us to modify a branched covering with removable Levy cycles so that it is no longer obstructed. This allows us to deal with cases where two polynomials are topologically mateable, but the formal mating has a Thurston obstruction. 

Let $F$ be a Thurston map. We need to set up a partial ordering on the (possibly empty) set $\RL(F)$ of isotopy classes of removable Levy cycles for $F$. First, recall from Definition~\ref{d:Levy} that if $\Lambda$ is a removable Levy cycle, then the set $\Sigma \setminus \Lambda$ consists of $n$ disks $D_1,\dotsc,D_n$ and an additional component $C$, such that for each $j$, the preimage $F^{-1}(D_j)$ has a component $D_j'$ isotopic to $D_{j-1}$ and the map $F \colon D_{j}' \to D_j$ is a homeomorphism. Also, by definition, for each $k$, all components of $F^{-k}(D_1)$ are disks. Define $\mathcal{D}(\Lambda) = \bigcup_{j=1}^n D_j$.

We now construct our partial order on $\RL(F)$. For two elements $\lambda_1, \lambda_2 \in \RL$, we say $\lambda_1 < \lambda_2$ if there exists representatives $\Lambda_1$ and $\Lambda_2$ of $\lambda_1$ and $\lambda_2$ respectively such that $\mathcal{D}(\Lambda_1) \subset \mathcal{D}(\Lambda_2)$. Define $\RL'$ to be the set of $\lambda \in \RL$ which are maximal under the partial order and such that for any Levy cycle $\Gamma$, the geometric intersection number  $\lambda \cdot \Gamma = 0$ (see e.g \cite[p.~28]{PrimerMCG} for the definition of geometric intersection number). In the case of matings, we have the following.

\begin{prop}[\cite{ShishTan}]\label{p:RLprime}
 Let $F = f \uplus g$ be the formal mating of two postcritically finite polynomials. The set $\RL'(F)$ consists of boundary curves of tubular neighbourhoods of periodic cycles $\{ [x_1],[x_2],\dotsc,[x_m]\}$ of ray classes such that each $[x_i]$ contains at least two points of $P_F$ and no ray class in $\bigcup_{n \geq 0} F^{-n}([x_i])$ contains a closed loop. Furthermore if $\Sigma / \approx$ is homeomorphic to $S^2$, then all Levy cycles for $F$ are removable. 
\end{prop}

The above result allows us to define (essential) mateability of two polynomials $f$ and $g$ in $\Sone$. This adapts the definition of mateability given in \cite{ShishTan} (Compare \cite{ShishEssMate}, where this was called the degenerate mating), which uses the notion of being \emph{weakly equivalent} to a rational map. One can think of this weak equivalence as being obtained by shrinking (simultaneously) all removable Levy cycles of the branched covering $F = f \uplus g$ to points. If the original two polynomials were topologically mateable, then this new branched covering will no longer have any Levy cycles.

\begin{defn}
 We say that two postcritically finite polynomials $f,g \in \Sone$ are \emph{(essentially) mateable} if all irreducible obstructions for $F = f \uplus g$ are removable Levy cycles.
\end{defn}

Since the matings of postcritically finite polynomials in $\Sone$ have hyperbolic orbifolds, the above definition is equivalent to the topological mating of $f$ and $g$ being Thurston equivalent to a rational map.

%%%%%%%%%%%%%%%%%%%%%%%%%%%%%%%%%%%%%%%%%%%%%%%%%%%%%%%

\section{Combinatorics of Limbs}

In this section, we collect together some combinatorial properties of the parameter limbs defined in Section~\ref{ss:paramlimbs}. 

\begin{lem}\label{l:fixedraylanding} Let $t_0 \neq 0$ be periodic under $m_2$ and suppose $F_a$ belongs to a limb with internal argument $t_0$.
\begin{enumerate}
\item  $F_a \in \Limb^+_{t_0}$, if and only if $R_a(0)$ lands on $\partial U_a$ and $R_a(1/2)$ lands on the critical limb of $F_a$.
\item $F_a \in \Limb^-_{t_0}$ if and only if the ray of angle $R_a(1/2)$ lands on $\partial U_a$ and $R_a(0)$ lands on the critical limb of $F_a$.
\end{enumerate}
\end{lem}

\begin{proof}
Notice that for any $F_a \in \mathcal{C}(\Sone)$ there exists a fixed point for $F_a$ on the boundary of $U_a$, and this point must be the landing point of either $R_a(0)$ or $R_a(1/2)$. Thus to prove the first claim, it suffices to show that $R_a(1/2)$ lands in the critical limb of $F_a$ whenever $F_a \in \Limb_{t_0}^+$.  In this case, since $t_0 \neq 0$, if $\theta < \theta'$ and $\pray(\theta)$ lands with $\pray(\theta')$  at the root of the parameter limb $\Limb^+_{t_0}$ then  $\{ \theta, \theta' \} \subset (5/6,1/6) \subset \T$. By Corollary~\ref{cor:critlimbrays}, the $R_a(\theta+1/3)$ and $R_a(\theta'-1/3)$ land on the root of the critical limb of $F_a$. But $\theta + 1/3 \in (1/6,1/2) \subset \T$ and $\theta'-1/3 \in (1/2,5/6) \subset \T$, and so $\theta + 1/3 < 1/2 < \theta' - 1/3$. Hence $R_a(1/2)$ must land in the critical limb of $F_a$. The proof for the second case is similar.
\end{proof}

In the special case that $t_0 = 0$, the external rays of angle $0$ and $1/2$ both land on  $\partial U_a$ at the root point of the critical limb of $F_a$. Together with their landing point, these rays separate the critical limb from the rest of Julia set.

Recall that for $F_a \in \Sone$, we denote by $\beta(t) = \Bot(e^{2 \pi i t})$ the point on $\partial U_a$ which is the landing point of the internal ray of angle $t$.

\begin{defn}\label{d:alphacycle}
Let $t_0 \in \Q / \Z$ be periodic of period $p$ under $m_2$ and suppose $F_a$ belongs to the limb $\Limb_{t_0}^{\pm}$. Then the points  $\beta(t_0),\, \beta(2t_0),\,  \dotsc, \,\beta(2^{p-1}t_0)$ form a period $p$ cycle under $F_a$ on the boundary of $U_a$, and these points are each the landing point of two external rays. This periodic cycle will be called the $\alpha$-\emph{periodic cycle} of $F$. The rotation number of the $\alpha$-periodic cycle is defined to be the rotation number of $t_0$ under angle doubling.
\end{defn}

The terminology is chosen to draw a parallel with the notion of the $\alpha$-periodic point of quadratic polynomials. By Proposition \ref{prop:limblanding}, this periodic orbit persists within a limb. That is, if $\Limb_{t_0}^\pm$ is a limb with periodic internal argument, then for all maps in this limb, the internal and external rays landing on the $\alpha$-periodic cycle have the same angles.  See Figure~\ref{f:alphacycle} for examples of two maps in the same limb. Thus when we refer to the rotation number of a limb $\Limb^\pm_{t_0}$, we mean the rotation number of the associated $\alpha$-periodic orbit. By the above discussion, this is well-defined. An equivalent definition is to call the $\alpha$-periodic cycle the unique periodic cycle of biaccessible points on the boundary of $U_a$.

\begin{figure}
\centering
\begin{subfigure}{0.5\textwidth}
\centering
\includegraphics[width=0.7\linewidth]{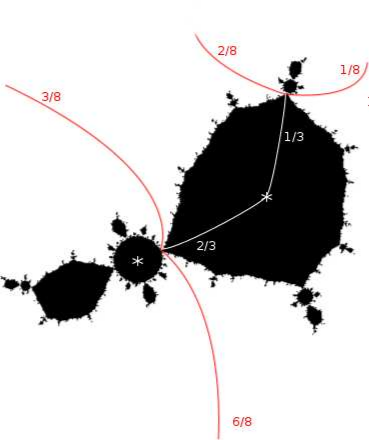}
\label{f:firstmap}
\end{subfigure}%
\begin{subfigure}{0.5\textwidth}
\centering
\includegraphics[width=0.9\linewidth]{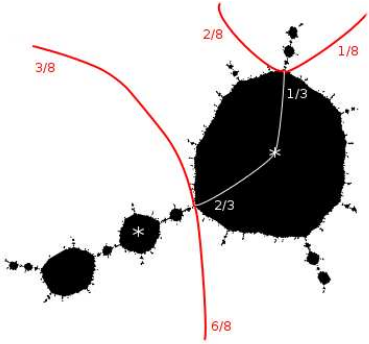}
\label{f:secondmap}
\end{subfigure}%
\caption{Two maps belonging to the limb $\Limb_{2/3}^+$ in $\Sone$. The rays (drawn by hand) landing on the $\alpha$-periodic cycle for both maps are shown. Note the angles of these rays are the same for both maps. The critical points are indicated by $\ast$. The figures show that $\Theta(\Limb_{2/3}^+) = \{ \frac18,\frac28,\frac38,\frac68\}$.} 
\label{f:alphacycle}
\end{figure}

We remark that we do not define the notion of an $\alpha$-periodic cycle when the internal argument of the limb is strictly preperiodic. Let $t_0$ be $m_2$-periodic and consider the limb $\Limb = \Limb_{t_0}^\pm$. Denote by $\Theta(\Limb)$ the set of all arguments of external rays landing on the $\alpha$-periodic cycle of the maps in $\Limb$. 

\begin{prop}\label{p:Theta_properties}
Let $t_0$ be periodic under $m_2$.
\begin{enumerate}
\item If the orbit $t_0$ under $m_2$ has rotation number $\rho$, then $\Theta(\Limb_{t_0}^{\pm})$ is a rotation set under $m_3$, with rotation number $\rho$. Furthermore, $\Theta(\Limb_{t_0}^{\pm})$ has exactly two major gaps.
\item If the orbit of $t_0$ under $m_2$ is not a rotation set, then $\Theta(\Limb_{t_0}^{\pm})$ is not a rotation set and has exactly one major gap.
\end{enumerate}
Furthermore, if $k$ is the period of $t_0$ under $m_2$, then for each $0 \leq i \leq k-1$, $\Theta(\Limb_{t_0}^{\pm})$ has a gap of length $\frac{3^i}{3^k-1}$. In particular, the two angles $\theta$ and $\theta'$ for which the corresponding rays land at the root of the critical limb of $F_a$ bound a major gap of length $\frac{3^{k-1}}{3^k-1}$.
\end{prop}

% MAYBE TALK ABOUT CRITICAL SECTORS AND CRITICAL VALUE SECTORS?

\begin{proof}
The proof follows easily from the fact that (external) rays cannot cross and Proposition~\ref{prop:rotsets}. For part (ii), the angles of the rays landing at the root of the critical limb bound the unique major gap. The existence of the gaps of the prescribed lengths follows from considering the forward images of the sectors bounded by the rays of angle $\theta$ and $\theta'$ in Proposition~\ref{prop:limblanding}. Details are left to the reader.
\end{proof}

We see that if $\Theta(\Limb)$ is a rotation set, then it is made up of the union of two distinct cycles under $m_3$, each with the same rotation number as $\Theta(\Limb)$. By Proposition~\ref{p:signaturetheorem}, such a set is completely described by its rotation number $p/q$ and $s_1(X)$, the number of elements of $X$ in $[0,1/2)$. Furthermore, since $s_1$ must be odd, we get the following.

\begin{lem}[Compare Corollary 3.15, \cite{SaeedBook}]\label{l:countinglimbs}
 Given a rotation number $\rho = p/q$, there exists exactly $q$ rotation sets with rotation number $p/q$ under $m_3$ which are the union of two cycles.
\end{lem}

\begin{proof}
The admissible values of $s_1$ are $1, 3, \dotsc 2q-1$, and each is associated to a unique rotation set.
\end{proof}

We now show that each of these $q$ rotation sets is realised as $\Theta(\Limb)$ for some limb $\Limb \subseteq \Sone$. 

\begin{cor}\label{cor:numofcomplimbs}
 Suppose $A$ is a rotation set under $m_3$ which is the union of two cycles and $\rho(A) = p/q$. Then there exists a unique limb $\Limb^+_{t_0}$ of $\Sone$ such that $A = \Theta(\Limb^+_{t_0})$  and a unique limb $\Limb_{t_0'}^-$ such that $A = \Theta(\Limb^-_{t_0'})$.
\end{cor}

We emphasise that $\Theta(\Limb)$ only consists of the angles of the external rays landing on the $\alpha$-periodic cycle for maps in $\Limb$. As noted above, the set $\Theta(\Limb)$ is the union of two distinct cycles $C_1$ and $C_2$ under $m_3$. If $F_a \in \Limb$, then each member of the $\alpha$-periodic cycle of $\Limb$ is the landing point of two external rays $R_a(\theta_1)$ and $R_a(\theta_2)$, where $\theta_1 \in C_1$ and $\theta_2 \in C_2$. However, if $A = \Theta(\Limb^+_{t_0}) = \Theta(\Limb^-_{t_0'})$ for some $t_0,t_0' \neq 0$, then the way the angles in $C_1$ and $C_2$ are paired (in the sense that the associated external rays have a common landing point) must be different. Indeed, by Lemma~\ref{l:fixedraylanding}, the ray $R_a(1/2)$ lands on the critical limb if $F_a \in \Limb_{t_0'}^+$. Thus for $F_a \in \Limb_{t_0'}^+$, if $\theta_1 = \max(A \cap (0,1/2))$ and $\theta_2 = \min(A \cap (1/2,1))$, then $R_a(\theta_1)$ and $R_a(\theta_2)$ land together. However, for $F_a \in \Limb^-_{t_0'}$, these rays cannot land together, since then they and their common landing point would separate $R_a(1/2)$ from $\partial U_a$, which contradicts Lemma~\ref{l:fixedraylanding}.

\begin{proof}[Proof of Corollary~\ref{cor:numofcomplimbs}]
 Since rotation cycles under $m_2$ are uniquely determined by their rotation number, we see that there are exactly $q$ limbs of the form $\Limb^+_{t_0}$ for which the internal argument $t_0$ belongs to the unique rotation set of $m_2$ that has rotation number $p/q$. By Lemma~\ref{l:countinglimbs}, there are $q$ candidates for $\Theta(\Limb)$ for such a limb. So it is sufficient to prove that if $t_0 \neq s_0$, then $\Theta(\Limb_{t_0}^+) \neq \Theta(\Limb_{s_0}^+)$. 
 
Assume $t_0$ belongs to a $m_2$-rotation set $X$. If $s_0$ does not belong to an $m_2$-rotation set, or belongs to a rotation set with a different rotation number than $X$, then by Proposition~\ref{p:Theta_properties} we must have $\Theta(\Limb_{t_0}^+) \neq \Theta(\Limb_{s_0}^+)$. So suppose $t_0,s_0 \in X$ but $t_0 \neq s_0$. Then the parameter rays landing at the root of $\Limb_{t_0}^+$ must be different than those landing at the root of $\Limb_{s_0}^+$. Hence, by Corollary~\ref{cor:critlimbrays}, the angles of the rays landing on the root of the critical limb of maps in $\Limb_{t_0}^+$ must be different from the angles of the rays landing on the root of the critical limb of maps in $\Limb_{s_0}^+$. But these angles bound a major gap for the rotation sets $\Theta(\Limb_{t_0}^+)$ and $\Theta(\Limb_{s_0}^+)$ respectively. Since the two rotation sets have different major gaps, it follows from Proposition~\ref{prop:gapthm} that $\Theta(\Limb_{t_0}^+) \neq \Theta(\Limb_{s_0}^+)$.

The proof for limbs of the form $\Limb_{t_0}^-$ is similar.
\end{proof}

% \begin{cor}
% If $A$ is a rotation cycle under $m_3$ and is the union of two distinct cycles, then there exactly two limbs of $\Sone$ for which $A = \Theta(\Limb)$.
% \end{cor}

% As the corollary shows, it is possible to have $\Theta(\Limb) = \Theta(\Limb')$ even when $\Limb \neq \Limb'$. For example, consider the limbs $\Limb^+$ and $\Limb^-$

Corollary~\ref{cor:numofcomplimbs} indicates that the dynamics of the limbs in $\Sone$ which have a rotation number are ``combinatorially full'', in the following sense. Given a rotation set $A$ under $m_3$, which consists of two periodic cycles, then there exists a limb $\Limb \subset \Sone$ for which $A = \Theta(\Limb)$. In fact, if one considers the moduli space $\Sone / \mathcal{I}$, where $\mathcal{I}$ is the canonical involution $z \mapsto -z$ (so that $\Sone / \mathcal{I}$ is made up of affine conjugacy classes of cubic polynomials with a fixed critical point), then there is precisely one limb corresponding to each rotation set $A$ of the form given above. This is analogous to the case for the Mandelbrot set; given a rotation set $A$ under $m_2$, there exists a limb in the Mandelbrot set for which the external rays landing on the associated $\alpha$-fixed point is the set $A$. However, it should be noted that, in contrast to the Mandelbrot set, there are limbs of $\Sone$ which don't have an associated rotation number. It would be interesting to study the combinatorics of the sets $\Theta(\Limb)$ for these limbs. Petersen and Zakeri have studied the combinatorics of periodic orbits under $m_d$ in \cite{Combtypes}.

We will use the notation $-\Theta(\Limb)$ for the set $\{ - \theta \mid \theta \in \Theta(\Limb) \}$.

\begin{defn}
Let $\Limb$ and $\Limb'$ be limbs in $\Sone$. We will say that $\Limb$ and $\Limb'$ are $\alpha$-\emph{symmetric} if $ \Theta(\Limb) = -\Theta(\Limb')$.
\end{defn}

Note that the above definition only makes sense for the case where $\Limb$ and $\Limb'$ have periodic internal argument. The importance of the above definition comes from the following result.

%By Lemma 2.5 of \cite{SaeedBook}, the two cycles in $\Theta(\Limb)$ are \emph{superlinked}; the points of the two cycles alternate as we travel around the circle. 

\begin{prop}\label{p:limbcondition}
 Let $\Limb_1$ and $\Limb_2$ be $\alpha$-symmetric limbs of $\Sone$. Then if $F_a$, $F_b$ are postcritically finite polynomials such that $F_a \in \Limb_1$ and $F_b \in \Limb_2$, then $F_a$ and $F_b$ are not mateable. 
\end{prop}

\begin{proof}
 Consider the ray class $R$ containing the $\alpha$-periodic cycle of $F_a$. Since $\Theta(\Limb_1) = -\Theta(\Limb_2)$, $R$ is also contains the $\alpha$-periodic cycle of $F_b$, and these points are the only elements of $J(F_a) \cup J(F_b)$ belonging to $R$. The points of an $\alpha$-periodic cycle are biaccessible, so this ray class may be thought of a (not necessarily connected) finite graph where every vertex has (at least) two edges landing on it. Such a graph must contain a closed loop. Thus the ray equivalence relation separates the sphere, and so by Proposition \ref{p:RLprime}, there exists a non-removable Levy cycle for the mating.
\end{proof}

\begin{rem}
 The notion of $\alpha$-symmetry again has close parallels to the quadratic case. More precisely, given a limb $\Limb$ within the Mandelbrot set, one could define $\Theta(\Limb)$ to be the angles of the external rays landing on the $\alpha$-fixed point of maps in $\Limb$ (compare Definition~\ref{d:alphacycle} and the following discussion). Then the condition that $ \Theta(\Limb) = -\Theta(\Limb')$ is exactly the case when $\Limb$ and $\Limb'$ are conjugate limbs of the Mandelbrot set. Thus by the result of Rees, Shishikura and Tan \cite{Tan:quadmat}, the mating of two quadratic postcritically finite polynomials is obstructed if and only if they belong to $\alpha$-symmetric limbs; compare Conjecture~\ref{conj:Conjecture}.
\end{rem}

Let $\Limb$ be a limb with periodic internal argument. In light of Proposition~\ref{p:limbcondition}, our goal is to now find combinatorial criteria for a limb $\Limb'$ which guarantee that $\Limb$ and $\Limb'$ are $\alpha$-symmetric. Since $\alpha$-symmetry is only defined for limbs with periodic internal argument, we will have to deal with limbs with strictly preperiodic internal arguments in a different manner. However, we remark that Theorem~\ref{t:preperneccessity} answers the mateability question in this case.

Given any limb $\Limb \subseteq \Sone$, we will say that the conjugate limb $\overline{\Limb}$ is the set of maps $F_a$ such that $F_{\overline{a}} \in \Limb$. In terms of internal arguments, if $t_0 \neq 0$ then the conjugate limb to $\Limb_{t_0}^{\pm}$ is the limb $\Limb_{-t_0}^\pm$. For $t_0 = 0$, the limbs $\Limb_{0}^+$ and $\Limb_0^-$ are mutually conjugate. In all cases, if the parameter ray $\pray(\theta)$ lands on the root of $\Limb$, then $\pray(-\theta)$ lands on the root of $\overline{\Limb}$.

\begin{lem}\label{l:conjugateTheta}
 Let $\Limb$ be a limb of $\Sone$ with periodic internal angle. Then $\Limb$ and $\overline{\Limb}$ are $\alpha$-symmetric.
\end{lem}

\begin{proof}
 Suppose  $\pray(\theta)$ and $\pray(\theta')$ (where $\theta < \theta'$) land on the root of $\Limb$. Then $\pray(-\theta)$ and $\pray(-\theta')$  land on the root of $\overline{\Limb}$. Hence by Corollary~\ref{cor:critlimbrays}, if $F_b \in \overline{\Limb}$ then the rays $R_b(-(\theta'-1/3))$ and $R_b(-(\theta+1/3))$ land on the root of the critical limb of $F_b$. But this means $-(\theta'-1/3)$ and $-(\theta+1/3)$ bound a major gap of $\Theta(\overline{\Limb})$ and in particular $\{ -(\theta'-1/3), -(\theta+ 1/3) \} \subseteq \Theta(\overline{\Limb})$. Taking forward iterates under $m_3$, it follows that $\Theta(\overline{\Limb}) = -\Theta(\Limb)$.
\end{proof}

In the special case where the internal argument is $0$, we have $\Theta(\Limb_0^+) = \Theta(\Limb_0^-) = \{0, 1/2 \}$.

\begin{prop}\label{p:conjugate}
Let $F_a$ and $F_b$ be postcritically finite polynomials that lie in conjugate limbs of $\Sone$. Then $F_a$ and $F_b$ are not mateable.
\end{prop}

\begin{proof}
If $\Limb$ and $\Limb'$ are conjugate limbs with periodic internal argument then $\Theta(\Limb) = - \Theta(\Limb')$ and so by Proposition~\ref{p:limbcondition}, the mating of $F_a$ and $F_b$ must be obstructed.

Now suppose $t_0$ is strictly preperiodic and $F_a$ is the unique element of $\Limb = \Limb^\pm_{t_0}$. Then $F_a$ is the landing point of a parameter ray $\pray(\theta)$, and the dynamical ray $R_a(\theta)$ lands on the cocritical point $2a \in J(F_a)$. This implies that the rays $R_a(\theta+1/3)$ and $R_a(\theta-1/3)$ land on the free critical point $-a \in J(F_a)$. Similarly if $F_b$ is the unique map in $\overline{\Limb}$, then the rays $R_b(-(\theta+1/3))$ and $R_b(-(\theta-1/3))$ land on the free critical point $-b \in J(F_b)$. But then the ray class under the mating containing the two critical points $-a$ and $-b$ contains a loop, and so the mating is obstructed.
\end{proof}

We provide some examples showing the obstructions that occur in the previous proposition. We use laminations (using only the leaves corresponding to the obstruction we are considering) to illustrate the ray classes that form an obstruction in the matings.

\begin{ex}\label{ex:4over7conj}
 Consider a map $F_a$ in the parameter limb $\Limb_{4/7}^+$. The parameter rays landing at the root of $\Limb_{4/7}^+$ have angles $\frac{1}{78}$ and $\frac{2}{78}$, from which it follows that the rays $R_a(1/26)$ and $R_a(2/26)$ land together in the dynamical plane of $F_a$. For a map $F_b$ in the conjugate limb $\Limb_{3/7}^+$, the rays $R_b(24/26)$ and $R_b(25/26)$ land together. Hence we have the lamination for the mating as given in Figure~\ref{f:examplefig}.
 \begin{figure}[ht]
    \centering
\begin{subfigure}{0.48\textwidth}
\centering
\includegraphics[width=0.7\linewidth]{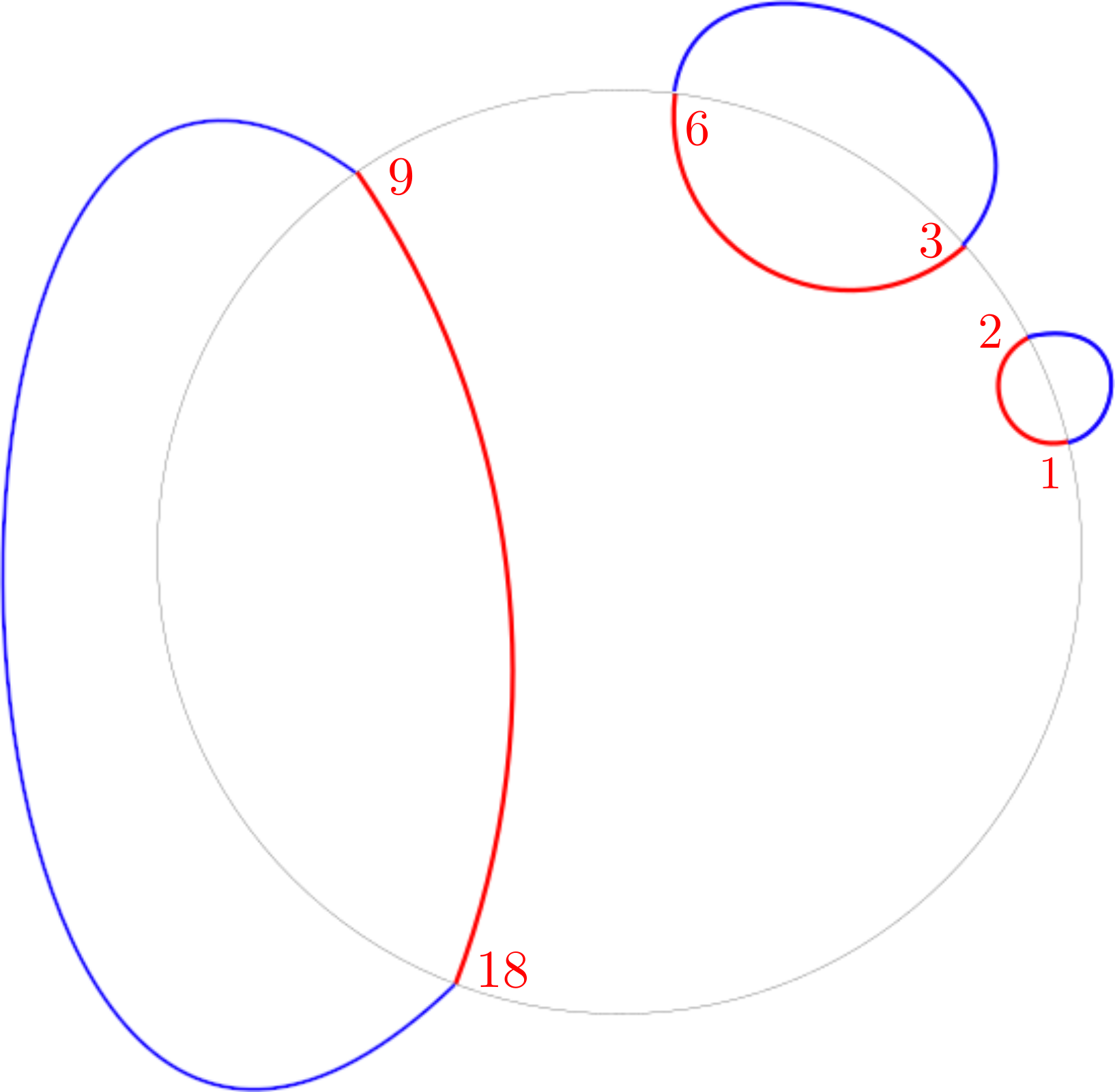}
\end{subfigure}\quad%
\begin{subfigure}{0.48\textwidth}
\centering
\includegraphics[width=0.7\linewidth]{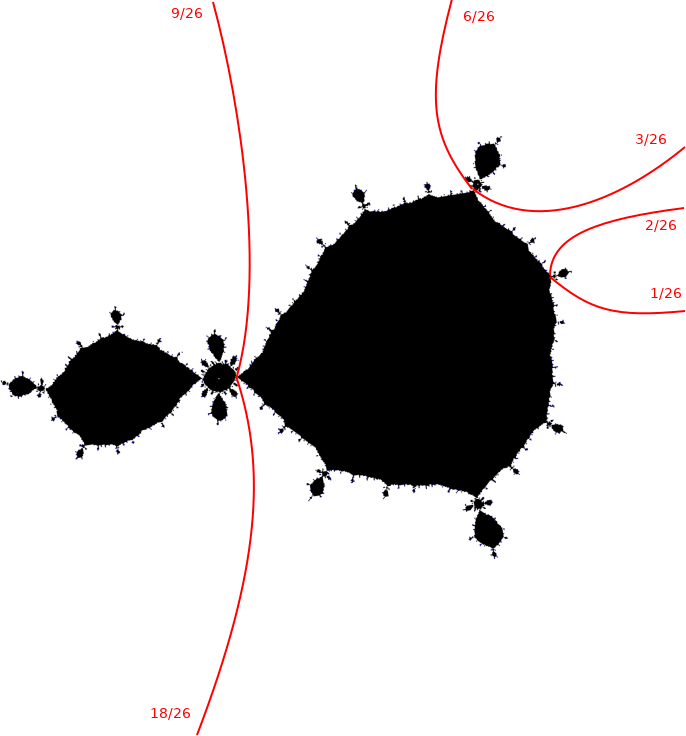}
\end{subfigure}
\caption{On the left is the lamination for the mating in Example~\ref{ex:4over7conj}. On the right is the Julia set for a possible ``inner map''.}\label{f:examplefig}
\end{figure}
\end{ex}

The following example shows what typically happens in the case when the limb $\Limb$ has strictly preperiodic internal argument.

\begin{ex}
We consider a case where the internal argument of the limb is strictly preperiodic, so that the limb consists of a single map on $\partial \Hy_0$. So let $F_a$ be the map in $\Limb_{5/6}^+$. This is the landing point of the ray $\pray(1/8)$, which means $R_a(11/24)$ and $R_a(19/24)$ land together at the free critical point $-a$. For the conjugate map $F_b$, which is the landing point of $\pray(7/8)$, the rays $R_b(7/24)$ and $R_b(13/24)$ both land on the free critical point $-b$.  The ray class containing $-a$ and $-b$ is a closed loop and so the mating is obstructed. See Figure~\ref{f:critfinmate} for the lamination of the mating and the Julia set of $F_a$.
\begin{figure}[!ht]
\centering
\begin{subfigure}{0.48\textwidth}
\centering
\includegraphics[width=0.7\linewidth]{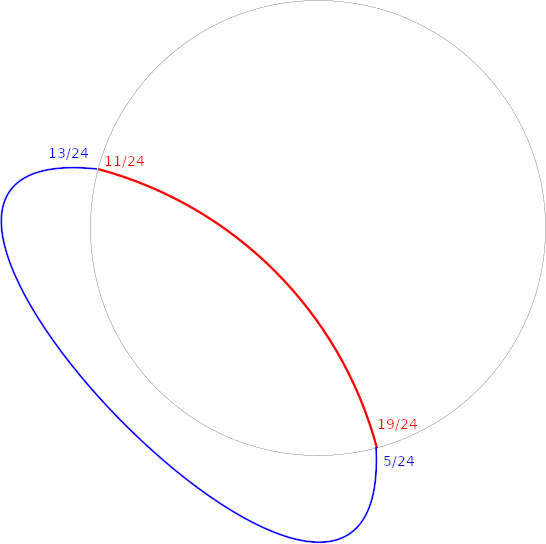}
\label{f:critfinobslam}
\end{subfigure}\quad%
\begin{subfigure}{0.48\textwidth}
\centering 
\includegraphics[width=0.7\linewidth]{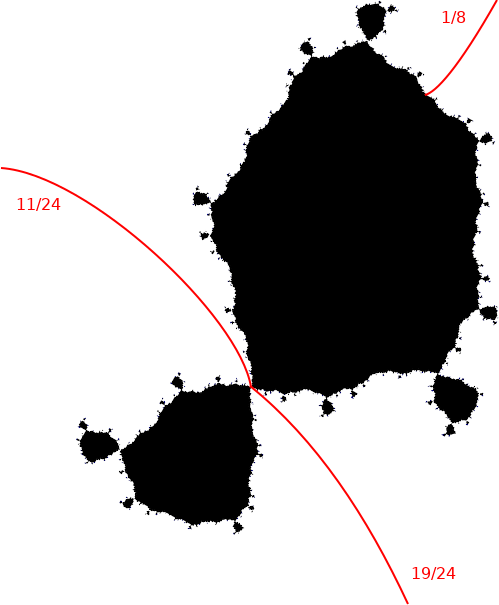}
\label{f:critfinpoly}
\end{subfigure}%
\caption{An obstructed mating for a map in a limb with strictly preperiodic internal argument. The two free critical points belong to the same ray class, and this ray class contains a closed loop. On the right is the Julia set for the ``inner map'' in the lamination.} 
\label{f:critfinmate}
\end{figure}
\end{ex}

In the first example where the internal argument is periodic each of the ray classes in the obstruction contain precisely one point from each of the $\alpha$-periodic cycles of $F_a$ and $F_b$. The obstruction is thus made up of $3$ closed loops, since the $\alpha$-periodic cycles of $F_a$ and $F_b$ both have period $3$. Furthermore, the ``smallest'' loop bounds a disk which contains both free critical values. When the internal argument is strictly preperiodic, the ray class containing a closed loop contains the free critical points of the polynomials. In this case, the Levy cycle generated (as guaranteed by Proposition~\ref{p:limset}) is a degenerate Levy cycle. In both cases, in the language of \cite{CubicObs}, we are in the ``quadratic-like'' case, where a disk component of the complement of the obstructing Levy cycle contains both free critical values.

\begin{defn}
We will say two limbs $\Limb$ and $\Limb'$ are \emph{complementary} if they are $\alpha$-symmetric but not conjugate.
\end{defn}
 
The limbs $\Limb_0^\pm$ are self-complementary; they are their own complementary limbs. Indeed is clear from the definition that in general limbs are mutually complementary; that is, $\Limb'$ is a complementary limb for $\Limb$ if and only if $\Limb$ is a complementary limb for $\Limb'$. By Corollary~\ref{cor:numofcomplimbs} and Lemma~\ref{l:conjugateTheta}, a complementary limb, if it exists, is unique. The following example shows that complementary limbs do exist. 

\begin{ex}
   Consider a map $F_a$ in the parameter limb $\Limb_{4/7}^+$. We saw before in Example~\ref{ex:4over7conj} that that the rays $R_a(\frac{1}{26})$ and $R_a(\frac{2}{26})$ land together in the dynamical plane of $F_a$. The complementary limb is $\Limb_{6/7}^-$. The associated external rays for $\Limb_{6/7}^-$ are $\pray(\frac{49}{78})$ and $\pray(\frac{50}{78})$, meaning if $F_b \in \Limb_{6/7}^-$ then the dynamical rays $R_b(\frac{23}{26})$ and $R_b(\frac{24}{26})$ land together in the dynamical plane. It is then not hard to see we have the lamination of the mating as given in Figure~\ref{firstex}.
   \begin{figure}[!ht]
    \centering
     \includegraphics[width=0.43\textwidth]{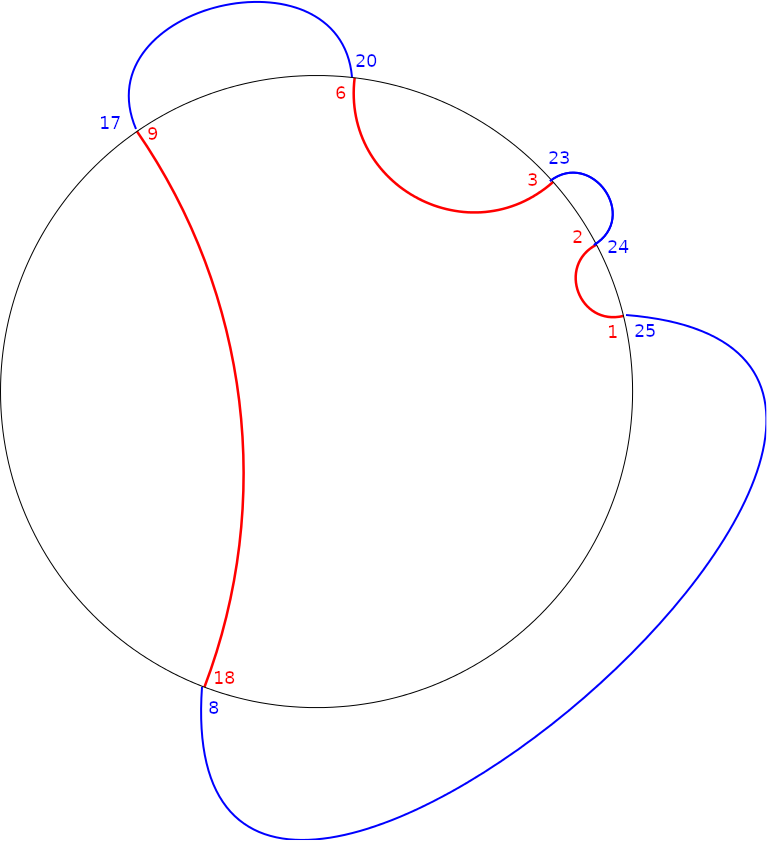}
    \caption{The lamination for the mating of a map in $\Limb_{4/7}^+$ and a  map in $\Limb_{6/7}^-$. The ray classes of the $\alpha$-periodic cycles form a single closed loop. The angles are labelled as multiples of $\frac{1}{26}$.} 
    \label{firstex}
   \end{figure}
   The ray class corresponding to the two $\alpha$-periodic cycles come together to form one large loop, thus the mating is obstructed by this Levy cycle.
 \end{ex}

% \begin{lem}\label{l:compobstructed}
% If $F_a$ and $F_b$ are postcritically finite polynomials which lie in complementary limbs of $\Sone$, then $F_a \uplus F_b$ is obstructed.
% \end{lem}
% 
% \begin{proof}
% If $\Limb$ and $\Limb'$ are complementary limbs then $\Theta(\Limb') = - \Theta(\Limb)$. Thus by  Proposition~\ref{p:limbcondition}, the mating of maps in complementary limbs of $\Sone$ is obstructed.
% \end{proof}

In the above example, there exists a disk component in the complement of the obstructing Levy cycle which contains a fixed critical point. In the language of \cite{CubicObs}, this corresponds to the ```Newton-like'' case.

In the next section, we give a combinatorial criterion for a limb $\Limb$ to have a complementary limb. 

\subsection{Combinatorial data of a limb}

We wish to prove the existence of complementary limbs for limbs whose internal argument has a rotation number. We will also give a combinatorial description of how a limb and a complementary limb are related.

\begin{prop}\label{p:complimbs}
Let $\Limb$ be a limb of $\Sone$. Then $\Limb$ has a complementary limb if and only if $\Limb$ has a rotation number.
\end{prop}

\begin{proof}
First suppose $\Limb$ does not have a rotation number. By Proposition~\ref{p:Theta_properties}, the set $\Theta(\Limb)$ has only one major gap, bounded by the angles $\theta$ and $\theta'$. In the dynamical plane for a map $F_a \in \Limb$, this major gap is bounded by the external rays $R_a(\theta)$ and $R_a(\theta')$ landing on the root point of the critical limb of $F_a$. Now suppose $\Limb'$ is a limb with $\Theta(\Limb') = - \Theta(\Limb)$. Then, since for any map $F_b \in \Limb'$ the rays landing at the root of the critical limb have to bound a major gap, we see that these rays must be angles $R_b(-\theta)$ and $R_b(-\theta')$. But then we see that $\Limb'$ must be the conjugate limb to $\Limb$, and so $\Limb$ has no complementary limb.

Now suppose $\Limb$ has a rotation number. Then the set $\Theta(\Limb)$ has two major gaps and thus the set $-\Theta(\Limb)$ also has two major gaps and so is a rotation set. But by Corollary~\ref{cor:numofcomplimbs}, there are two limbs satisfying $\Theta(\Limb') = - \Theta(\Limb)$. Since only one can be a conjugate limb, we see that $\Limb$ must have a complementary limb.
\end{proof}

We provide a combinatorial description of parameter limbs which will allow us to find the complementary limb for any limb which has a rotation number. We first need to discuss the combinatorics of rotation sets under $m_2$. Let $t_0$ belong to the unique rotation set $X_{p/q}$ under $m_2$ with rotation number $\rho = p/q \in \Q-\{0\}$. Write $X_{p/q}$ in cyclic order as $\{ t_1, t_2, \dotsc, t_q \}$, so that $0 \in (t_q,t_1)$. Then there exists a unique $1 \leq k \leq q$ such that $t_0 = t_k$. We call $k$ the (combinatorial) position of $t_0$ in $X_{p/q}$. We define the \emph{combinatorial data} for the limb $\Limb_{t_0}^{\pm}$ to be the triple $(\rho,k,\pm)$.

\begin{ex}
Consider the $m_2$-rotation set $X_{2/5}$ with rotation number $\rho = \frac{2}{5}$. Then $X_{2/5} =\{ 5/31,9/31,10/31,18/31, 20/31  \}$ and the combinatorial position of $18/31$ is $4$. Thus the Limb $\Limb_{18/31}^+$ has combinatorial data $\left( \frac{2}{5}, 4, + \right)$.
\end{ex}

It is clear that, given a triple $\mathcal{D} = (p/q,k,\pm)$ where $1 \leq k \leq q$ , there exists a unique limb $\Limb$ which realises the combinatorial data $\mathcal{D}$. The combinatorial data above allows us to easily recognise complementary limbs. As an aside, we remark that if a limb has combinatorial data $(p/q,k,\pm)$ then the conjugate limb has combinatorial data $(-p/q,q+1-k,\pm)$. 

\begin{lem}\label{l:positioncharacterisation}
Suppose the combinatorial position of $t_0$ is $k$. 
\begin{enumerate}
\item  The number of elements in $\Theta(\Limb_{t_0}^+) \cap [0,1/2)$ is equal to $2k-1$.
\item  The number of elements in $\Theta(\Limb_{t_0}^-) \cap [1/2,1)$ is equal to $2k-1$.
\end{enumerate}
\end{lem}

\begin{proof}
As usual, we only prove the first claim, since the proof of the second is similar. For $t_0 = 0$, the only possible position is $k=1$ and the result follows since the only rotation set consisting of two cycles with rotation number $0/1$ is $\{0,1/2\}$. Suppose $t_0 \neq 0$ has combinatorial position $k$ and let $F_a \in \Limb_{t_0}^+$. Then by Lemma~\ref{l:fixedraylanding}, the angles of the external rays $R_a(\theta)$, $R_a(\theta')$ landing on the root of the critical limb of $F_a$ satisfy $\theta < 1/2 < \theta'$. Since $\beta(0)$ is the landing point of the ray $R_a(0)$, then for $1 \leq i \leq k-1$, the two external rays landing on $\beta(t_i)$ lie in $[0,1/2)$, and for $i > k$, the angles of the external rays landing on $\beta(t_i)$ lie in $[1/2,1)$. The result follows. 
\end{proof}

\begin{lem}\label{l:limbdata}
 Suppose $\Limb$ is the limb with combinatorial data $(p/q,k,\pm)$. Then the limb with data $\mathcal{D}'=(-p/q,k,\mp)$ is the complementary limb to $\Limb$.
 \end{lem}

 \begin{proof}
  The result clearly holds for $t_0$ = 0. So suppose $\Limb = \Limb_{t_0}^+$ has data $(p/q,k,+)$ with $p/q \neq 0$. The complementary limb $\Limb'$, which we know exists by Proposition~\ref{p:complimbs}, must then satisfy $\Theta(\Limb') = - \Theta(\Limb)$ and so will have rotation number $-p/q$. By Corollary~\ref{cor:numofcomplimbs}, there are two limbs for which this equality holds. One must be the conjugate limb to $\Limb$, which is the limb $\Limb_{-t_0}^+$, and the other must be a limb $\Limb_{s_0}^-$ for some $s_0$. 
  
  By Lemma~\ref{l:positioncharacterisation}, there are $2k-1$ elements in $\Theta(\Limb) \cap (0,1/2)$. Thus the complementary limb $\Limb_{s_0}^-$ must have $2k-1$ elements in $\Theta(\Limb_{s_0}^-) \cap (1/2,1)$. But again by Lemma~\ref{l:positioncharacterisation}, this means the combinatorial position of $s_0$ is $k$. Hence the complementary limb has data $(-p/q,k,-)$. The proof for a limb with data $(p/q,k,-)$ is analogous.
 \end{proof}

 We give some examples of how we find complementary limbs using the above lemma.

 \begin{ex}
  Again we consider the limb $\Limb_{18/31}^+$. Since this limb has combinatorial data $\left( \frac{2}{5}, 4, + \right)$, Lemma~\ref{l:limbdata} implies that the complementary limb has combinatorial data $\left( \frac{3}{5}, 4, - \right)$. Let $X_{3/5}$ be the unique rotation set under $m_2$ with rotation number $3/5$. Then
  \[
   X_{3/5} = -X_{2/5} = \{ 11/31, 13/31, 21/31, 22/31, 26/31 \}
  \]
The element in this set with position $4$ is $22/31$. Thus the complementary limb to $\Limb_{18/31}^+$ is $\Limb_{22/31}^-$. Using results of Zakeri \cite{SaeedBook}, one can compute that the rays $\pray(19/726)$ and $\pray(20/726)$ land at the root of $\Limb_{18/31}^+$, while the rays $\pray(427/726)$ and $\pray(428/726)$ land on the root of $\Limb_{22/31}^-$. The major gaps for $\Theta(\Limb_{18/31}^+)$ are the arcs $( 87/242,  168,242)$ and $(180/242,19/242)$. The major gaps for $\Theta(\Limb_{22/31}^-$ are the arcs $(74/242,155/242)$ and $(223/242,62/242)$. One can easily check that the ray class formed create a single closed loop.
 \end{ex}
 
  Before giving the next example, we present some results from \cite{ACUB}. Let $F_a \in \Sone$ and denote by $F_a^\ast$ the map $F_{-\overline{a}}$. It was shown that, if the mating $F_a \mate F_a^\ast$ was not obstructed, then the resulting rational map belongs to the space $\mathcal{A}$ of antipode preserving cubic rational maps (That is, maps which commute with the map $z \mapsto -1/\overline{z}$) which have a pair of fixed critical points. In particular, if $F_a \in \Hy_0 \subseteq \Sone$, then it was shown that the mating $F_a \mate F_a^\ast$ exists and belongs to the main hyperbolic component $\widetilde{\Hy}_0$ in the space of $\mathcal{A}$ (this component has the map $z \mapsto -z^3$ as its centre). As with $\Hy_0 \subseteq \Sone$, one can assign an internal argument to points on $\partial \widetilde{\Hy}_0$.
  
  Now suppose that $\Theta_{p/q}$ is the (unique) $m_2$-rotation set with rotation number $p/q$. When $q$ is odd, there exists a unique $\theta \in \Theta_{p/q}$ which has combinatorial position $(q+1)/2$ in $\Theta_{p/q}$. Following \cite{ACUB}, we say the angle $\theta$ is \emph{balanced}. It was shown that if an angle $\theta$ is balanced, then the corresponding point on $\partial \widetilde{\Hy}_0$ with internal argument $\theta$ is an ideal point on the circle at infinity in $\mathcal{A}$. This suggests that matings of maps in the $\Limb_\theta^\pm$ with those in the limb $(\Limb_\theta^\pm)^\ast $ is obstructed. But observe that if $\theta \in \Theta_{p/q}$ is balanced then the combinatorial position of $\theta$ is $(q+1)/2$ and so $\Limb_\theta^+$ has data $(p/q,(q+1)/2,+)$. Lemma~\ref{l:limbdata} then asserts that the complementary limb will have data $(-p/q,(q+1)/2,-)$. But the limb with this data is $\Limb_{-\theta}^-$, which is precisely the limb $(\Limb_\theta^\pm)^\ast $. Thus the balanced case corresponds exactly to the case where $\Limb^\ast$ is the complementary limb to $\Limb$. We give an example below.
 
  \begin{ex}
   Consider the limb $\Limb_{5/7}^+$. The associated external (parameter) angles are $\frac{7}{78}$ and $\frac{8}{78}$, and the combinatorial data of $\Limb$ is $\left(\frac{2}{3},2,+\right)$, and so the angle $5/7$ is balanced. For maps in this limb, the rays of angles $\frac{7}{26}$ and $\frac{8}{26}$ land at a common period $3$ point on $\partial U_a$. The limb with combinatorial data $\left( \frac{1}{3}, 2,-\right)$ has internal angle $\frac{2}{7}$ and the associated external angles are $\frac{31}{78}$ and $\frac{32}{38}$, from which it follows that the rays of angles $\frac{5}{26}$ and $\frac{6}{26}$ land together in the dynamical plane. This time we have situation as illustrated in Figure~\ref{f:secondex}.
   \begin{figure}[ht]
    \begin{center}
     \includegraphics[width=0.35\textwidth]{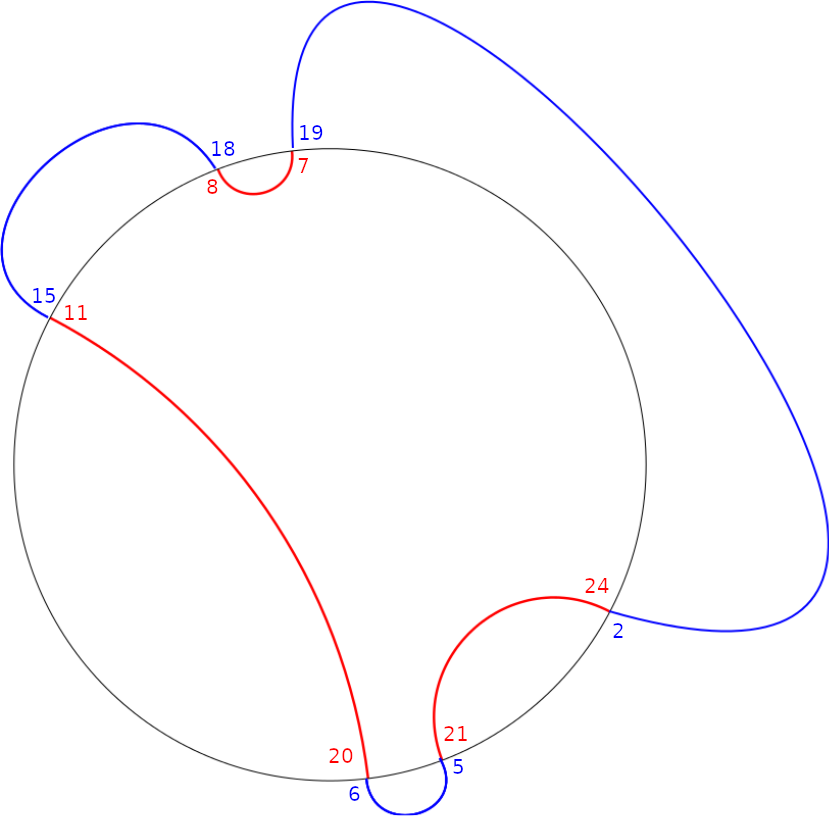}
     \caption{An example when the limb is balanced in the sense of \cite{ACUB}. Notice that picture has a symmetry (inversion in the unit circle followed by $180^\circ$ rotation), which occurs because the complementary limb is equal to the ``negative conjugate'' of the original limb in the balanced case. We again label the angles as multiples of $\frac{1}{26}$.}
     \label{f:secondex}
    \end{center}
   \end{figure}

   Again, the ray class corresponding to the $\alpha$-periodic cycles form one loop.
   \end{ex}

\section{Main results}

%%%%% MAYBE RENAME THIS %%%%

We now are ready to state the main theorem. Recall that in the quadratic case, the culmination of work by Rees, Shishikura and Tan gave the following.

\begin{thm}[\cite{Tan:quadmat}]
Let $f$ and $g$ be postcritically finite quadratic polynomials in the Mandelbrot set. Then $f \mate g$ is obstructed if and only if $f$ and $g$ belong to conjugate limbs of the Mandelbrot set.
\end{thm}

%The following shows that the analogue to the sufficiency result also holds in $\Sone$. Later we will show that for limbs with strictly preperiodic internal argument, we can also prove necessity. However, it is certainly not true that the conjugate limb condition is necessary in general for matings of postcritically finite polynomials in $\Sone$. 
%
%By Proposition~\ref{p:conjugate}, the sufficiency of this result also holds for matings of postcritically finite polynomials in $\Sone$. Our main result shows that necessity does not hold. We also suggest a conjectural necessary condition, based on the combinatorics of limbs.

The work carried out so far means we can easily prove the following, which is a slightly stronger version of Theorem~\ref{introthm}.

\begin{thm}\label{MainThm}
Let $t_0$ be $m_2$-periodic such that the associated orbit has a rotation number. Let $F_a$ be a postcritically finite polynomial in the limb $\Limb_{t_0}^{\pm}$. Then there exists a postcritically finite polynomial $F_b$ which does not belong to the conjugate limb $\Limb_{-t_0}^{\pm}$ such that $F_a$ and $F_b$ are not mateable.
\end{thm}

\begin{proof}
Since $t_0$ has a rotation number under $m_2$, then by Proposition~\ref{p:complimbs}, the limb $\Limb_{t_0}^{\pm}$ has a complementary limb $\Limb'$. Given $F_a \in \Limb_{t_0}^{\pm}$, take $F_b$ to be any postcritically finite map in $\Limb'$. Then $F_a \uplus F_b$ is obstructed.
\end{proof}

The proof of Theorem~\ref{introthm} follows immediately. We believe that $\alpha$-symmetry is also necessary, and so conjecture the following.

%The map $F_b$ constructed in the proof of the above theorem will belong to the complementary $\Limb_{t_0}^{\pm}$ which contains $F_a$. 

%In fact, we will prove the following, which implies Theorem~\ref{MainThm}.

% The following conjecture states that the limb conditions given in Propositions~\ref{p:conjugate} and \ref{p:complementary} are both necessary and sufficient. 

 \begin{conj}\label{conj:Conjecture}
 Let $f,g \in \Sone$ be postcritically finite polynomials in $\Sone$. Then $f$ and $g$ are not mateable if and only if $f$ and $g$ belong to conjugate or complementary limbs in $\Sone$.
 \end{conj}

Thus, we conjecture that the mateability question in $\Sone$ can be completely understood from the combinatorics of the $\alpha$-periodic cycles of the two polynomials, just as the mateability question for quadratic polynomials can be completely understood from the combinatorics of the rays landing on the $\alpha$-fixed points of the polynomials being mated. Numerical experiments, involving looking at various parts of the space of cubic rational maps which have two fixed critical points (for published examples, see \cite{Baranski,ACUB}) provide strong evidence that this is the case.  

   In the quadratic case, the mating of $f$ and $g$ is obstructed if and only if the ray classes under the mating of the $\alpha$-fixed points of $f$ and $g$ coincide. The generalisation for the $\alpha$-periodic cycles does not hold: it is possible for the two $\alpha$-periodic cycles to belong to the same ray class (or a union of ray classes, each of which contains an element of both $\alpha$-periodic cycles), but for this ray class (or union of ray classes) not to contain a loop. In terms of the sets $\Theta(\Limb)$ and $\Theta(\Limb')$, this means that $\Theta(\Limb) \cap -\Theta(\Limb') \neq \varnothing$ but $\Theta(\Limb) \neq -\Theta(\Limb')$. The reason is that $m_2$-rotation sets are uniquely defined by their rotation numbers; this is not the case for $m_3$-rotation sets.
   
   \begin{ex}\label{ex:hypnotobstructed}
    We give an example where $\Theta(\Limb) \cap -\Theta(\Limb') \neq \varnothing$ but $\Limb$ and  $\Limb'$ are not $\alpha$-symmetric. Let $F_a$ be the unique map in $\Limb = \Limb_{2/3}^+$ for which the free critical point $-a$ has period $2$ (the Julia set is shown on the left of Figure~\ref{f:alphacycle}). Since $\Theta(\Limb) = \{ 1/8,2/8,3/8,6/8\}$, we see that $\Theta(\Limb) \cap -\Theta(\Limb) = \{ 2/8,6/8\}$. But the mating is not obstructed; the only biaccessible periodic points for $F_a$ belong to the $\alpha$-periodic cycle, and so there is no Levy cycle for this mating. $F_a \uplus F_a$ is equivalent to the rational map given in Figure~\ref{f:mating}.
    
\begin{figure}[ht]
\centering
\begin{subfigure}{0.4\textwidth}
\centering
\includegraphics[width=0.7\linewidth]{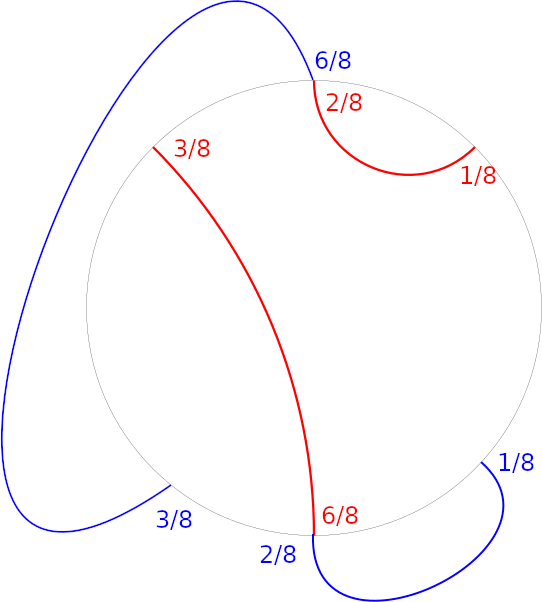}
\label{f:selfmatelam}
\end{subfigure}\quad%
\begin{subfigure}{0.57\textwidth}
\centering
\includegraphics[width=0.9\linewidth]{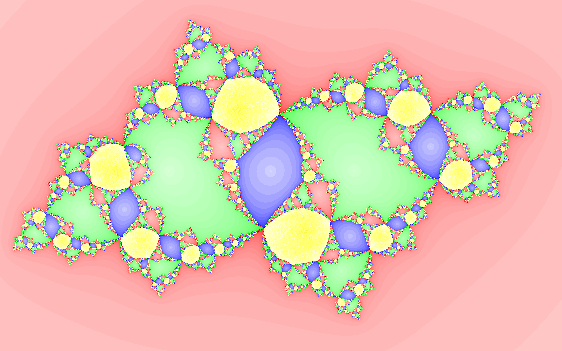}
\label{f:selfmate}
\end{subfigure}%
\caption{An example where $\Theta(\Limb) \cap -\Theta(\Limb') \neq \varnothing$ but the mating is not obstructed. On the left is the lamination, with labelled angles, showing the rays landing on the $\alpha$-periodic cycles, and on the right is the Julia set of the rational map constructed by the mating.} 
\label{f:mating}
\end{figure}
   \end{ex}

 \subsection{Necessity of $\alpha$-symmetry for preperiodic internal arguments} 
 In the case where the internal argument $t_0$ of the limb $\Limb = \Limb_{t_0}^\pm$  is strictly preperiodic, we can prove that if the mating $F_a \uplus F_b$ of postcritically finite $F_a \in \Limb$ and $F_b \in \Limb'$ is obstructed, then $\Limb$ and $\Limb'$ are $\alpha$-symmetric. The goal of this section is to prove the following theorem.
  
 \begin{thm}\label{t:preperneccessity}
  Let $t_0$ be strictly preperiodic under $m_2$ and suppose $F_a$ is the unique element of the limb $\Limb_{t_0}^\pm$. If $F_b$ is a postcritically finite polynomial in $\Sone$ then $F_a$ and $F_b$ are not mateable if and only if $F_b$ is the unique element of the conjugate limb $\Limb_{-t_0}^\pm$.
 \end{thm}
 
 The following lemma contains important properties of the ray classes in the matings considered in the proof of Theorem~\ref{t:preperneccessity}.
 
 \begin{lem}\label{l:v_boppositev_a}
  Let $t_0$ be strictly preperiodic under $m_2$ and suppose $F_a$ is the unique element of the limb $\Limb_{t_0}^\pm$. Let $F_b$ be a postcritically finite polynomial in $\Sone$ such that, in the mating $F_a \uplus F_b$, the free critical values $v_a$ and $v_b$ of $F_a$ and $F_b$ respectively belong to the same ray class $[v_a]$. Then:
  \begin{enumerate}
   \item if $R_{a}(\theta)$ is the ray landing on $v_a$, then $R_{b}(-\theta)$ lands on $v_b$.
   \item there exists an angle $\eta$ such that $R_a(\eta)$ lands on the free critical point $-a$ of $F_a$ and $R_b(-\eta)$ lands on the free critical point $-b$ of $F_b$.
   \item $v_b$ is the only point in the ray class $[v_a]$ which belongs to $J(F_b)$.
   \item the free critical points $-a$ and $-b$ of $F_a$ and $F_b$ respectively belong to the same ray class $[-a]$. 
   \item the ray class $[-a]$ containing the two free critical points contains a closed loop if and only if there exist distinct $\eta$, $\eta'$ such that $R_a(\eta)$ lands together with $R_a(\eta')$ on $-a$ and $R_b(-\eta)$ lands together with $R_b(\eta')$ on $-b$.
   \end{enumerate} 
 \end{lem}
 
 \begin{proof}\mbox{}
 \begin{enumerate}
  \item Suppose $R_{F_b}(-\theta)$ does not land on $v_b$. Since $R_{a}(\theta)$ is the unique external ray landing on $v_a$, it follows that there must be a biaccessible point $z$ in $[v_a] \cap J(F_a)$. By assumption, the free critical point $-a$ of $F_a$ is not periodic, and so the ray class $[v_a]$ maps forward homeomorphically under iterates of the mating. In particular, there is a forward iterate of $z$ which is a biaccessible periodic point in $J(F_a)$. But since $F_a$ has no biaccessible periodic point, this is a contradiction.
  \item Let $R_a(\theta)$ be the ray landing on $v_a$. By part (i), $R_b(-\theta)$ lands on $v_b$. Consider the three preimages of $R_a(\theta)$ under $F_a$. Two of these preimage rays must land on $-a$. Similarly, two of the three preimages of $R_b(-\theta)$ must land on $-b$. By the pigeonhole principle, there must then exist some $\eta$ such that $R_a(\eta)$ lands on $-a$ and $R_b(-\eta)$ lands on $-b$.
\item  To see that $v_b$ is the only element of the ray class that belongs to $J(F_b)$, first observe that if there is another element of $J(F_b)$ in the ray class, then there must also be a biaccessible point of $J(F_a)$ in the ray class. Similarly to the above, we can then map the ray class forward and obtain a biaccessible periodic point in $J(F_a)$. Again this is a contradiction.
\item Suppose that $[-a]$ and $[-b]$ are distinct ray classes. Both of these ray classes map onto $[v_a]$ by a degree $2$ mapping, and so elements of $[v_a]$ have at least four preimages under $F_a \uplus F_b$. But the mating has degree $3$, so this is a contradiction and so the ray classes $[-a]$ and $[-b]$ are equal. 
\item If there exist $\theta$ and $\theta'$ such that $R_a(\theta)$ and $R_a(\theta')$ land on $-a$ and $R_b(-\theta)$ and $R_b(\theta')$ land on $-b$, then (the closure of) these rays form a closed loop within the ray class $[-a]$. 

Now suppose that there is only one angle $\eta$ such that $R_a(\eta)$ lands on $-a$ and $R_b(\eta)$ lands on $-b$ (there must be at least one such $\eta$ by part (ii)). Then we see that both $-b$ and $2b$ both belong to $[-a]$, and these are the only two elements in $[-a] \cap J(F_b)$, since $v_b$ is the only element of $J(F_b)$ in $[v_a]$ (part (iii)). Furthermore, suppose $z \in J(F_a) \cap [-a]$ is an element distinct from $-a$. We claim $z$ is the landing point of only one external ray. For if not, similar reasoning as that used in part (i) would imply that a forward iterate of $z$ is a biaccessible periodic point for $F_a$, which is impossible.

 We now think of the ray class $[-a]$ as a connected bipartite graph with parts $A = [-a] \cap J(F_a)$ and $B = [-a] \cap J(F_b) = \{-b,2b\}$, and such that there is an edge between $x \in A$ and $y \in B$ if and only if there exists $t$ such that $R_a(t)$ lands on $x$ and $R_b(-t)$ lands on $y$. Suppose the number of vertices in this graph is $n$. Then since $B = \{-b,2b\}$,we see there are $n-2$ elements of $A$. Since the graph is bipartite, the number of edges in the graph is equal to the sum of the degrees (in the graph theory sense) of the vertices in $A$. But $-a$ has degree $2$, and it was shown above that each of the other $n-3$ vertices in $A$ has degree $1$. Thus there are $n-1$ edges in the graph. Since a connected graph with $n$ vertices is a tree if and only it has $n-1$ edges, we see that the graph does not contain a closed loop.
\end{enumerate}
 \end{proof}
 
%NEED CASES FOR SECOND PART OF RESULT. WE KNOW RAY CLASSES OF CRITICAL VALUES COINCIDE, BUT THEN WHAT?
%1. 3t lands on v_a and then -3t lands on v_b
%2. 3t lands on v_a but -3t does not land on v_b. Then get biaccessible z in J(F_a) in the ray class, but forward image must be biaccessible periodic, contradiction.

 \begin{proof}[Proof of Theorem~\ref{t:preperneccessity}]
 If $F_b$ is the conjugate map to $F_a$ then the mating is obstructed by Proposition~\ref{p:conjugate}.
 
 To prove the converse, we will show that if $F_b$ is not the conjugate map to $F_a$, then any obstruction to the mating $F_a \uplus F_b$ is a removable Levy cycle. First note that by Corollary \ref{c:postcritfinmatings}, any obstruction to this mating must be a degenerate Levy cycle, and if this Levy cycle is not removable, it follows from Corollary~\ref{cor:twocritvals} that there must be a ray class which contains both $v_a$ and $v_b$, the free critical values of $F_a$ and $F_b$ respectively. Let $\pray(\theta)$ be the unique parameter ray landing on $F_a$. Then $R_a(\theta)$ is the only ray landing on the cocritical point $2a \in J(F_a)$. Taking the forward image, we see that $R_a(3\theta)$ is the only ray which lands on the free critical value $v_a$ of $F_a$ and so $R_a(\theta-1/3)$ and $R_a(\theta+1/3)$ are the only rays landing on the free critical point $-a$. According to Lemma~\ref{l:v_boppositev_a} (part (i)), since $v_a$ and $v_b$ to belong to the same ray class $[v_a]$, then $v_b$ is the landing point of $R_b(-3\theta)$. This means that the cocritical point $2b$ of $F_b$ must be the landing point of one of $R_b(-\theta)$, $R_b(-(\theta-1/3))$ or $R_b(-(\theta+1/3))$. Hence, if $v_a$ and $v_b$ belong to the same ray class, then $F_b$ is equal to the landing point of one of $\pray(-\theta)$, $\pray(-(\theta-1/3))$ or $\pray(-(\theta+1/3))$. Furthermore, again by Lemma~\ref{l:v_boppositev_a} (part (iii)), the only element of $J(F_b)$ in $[v_a]$ is $v_b$, and the preimage of the ray class $[v_a]$ contains a ray class $[-a]$ which contains both free critical points (part (iv)).
 
 If $F_b$ is not the conjugate map to $F_a$, then $F_b$ is not the landing point of $\pray(-\theta)$. Thus it remains to show that if $F_b$ is the landing point of $\pray(-(\theta-1/3))$ or $\pray(-(\theta+1/3))$, then no ray class in the mating $F_a \uplus F_b$ which contains a closed loop. In fact, it suffices to show that this is the case for the ray class $[-a]$ which contains both $-a$ and $-b$, since all other ray classes map homeomorphically onto their image, and the ray class $[v_a]$ does not contain a closed loop by Proposition~\ref{p:limset}. We will assume $F_b$ is the landing point of $\pray(-(\theta-1/3))$; the proof of the other case is similar. We may also assume $F_b$ is not parabolic, since if $F_b$ is a parabolic map then $v_b$ does not belong to the Julia set of $F_b$.

Thus $R_b(-(\theta-1/3))$ lands on the cocritical point $2b$ and therefore the rays $R_b(-\theta)$ and $R_b(-(\theta+1/3))$ land on the free critical point $-b$. Since $R_a(\theta)$ lands on $2a$ and $R_a(\theta-1/3)$ and $R_a(\theta+1/3)$ are the only rays which land on $-a$, we see that there is only the angle $\eta = \theta + 1/3$ for which $R_a(\eta)$ lands on $-a$ and $R_b(-\eta)$ lands on $-b$. By Lemma~\ref{l:v_boppositev_a} (part (v)), the ray class does not contain a closed loop, and so $F_a$ and $F_b$ are mateable.
\end{proof}
 
  \begin{figure}[!ht]
\centering
\begin{subfigure}{0.57\textwidth}
\centering
\includegraphics[width=0.88\linewidth]{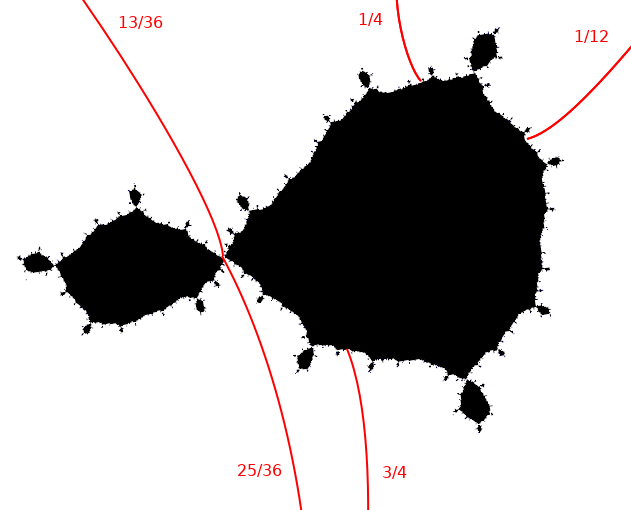}
\label{f:1over72}
\end{subfigure}\quad%
\begin{subfigure}{0.4\textwidth}
\centering
\includegraphics[width=0.87\linewidth]{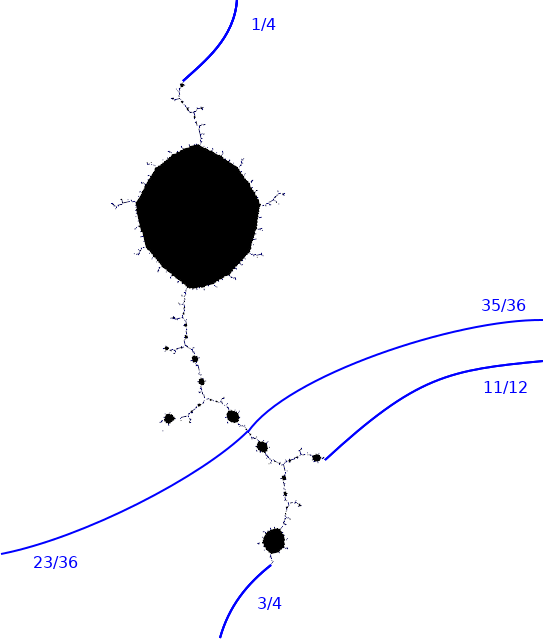}
\label{f:22over72}
\end{subfigure}%

\begin{subfigure}{0.4\textwidth}
\centering
\includegraphics[width=0.8\linewidth]{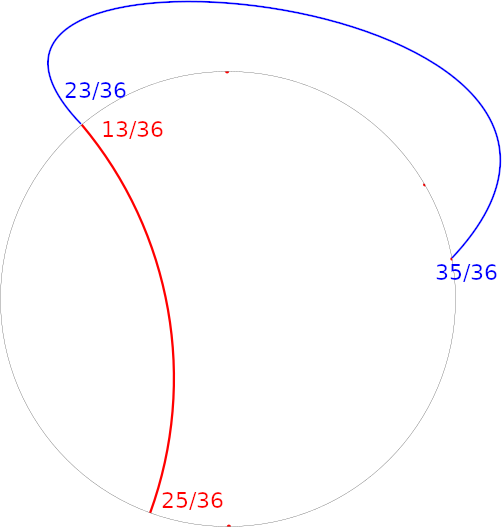}
\label{f:doublecritex}
\end{subfigure}\quad%
\begin{subfigure}{0.5\textwidth}
\centering
\includegraphics[width=0.9\linewidth]{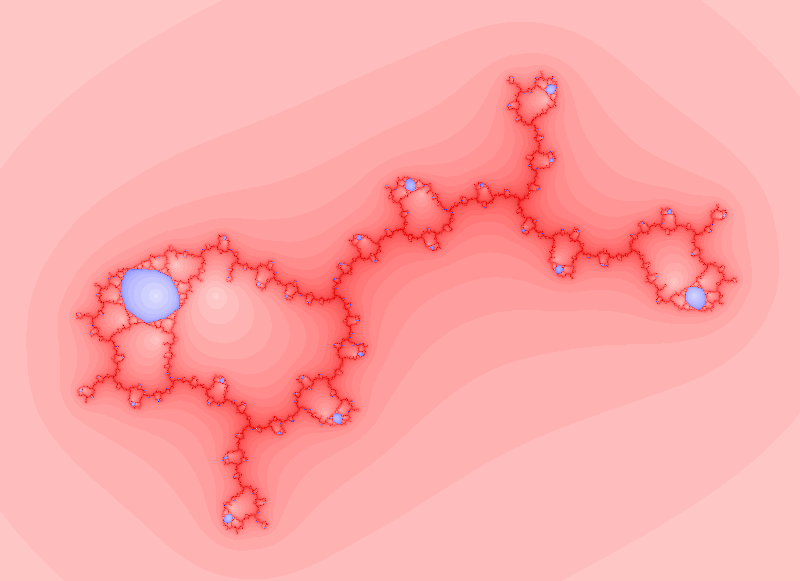}
\label{f:cubicratmap}
\end{subfigure}%
\caption{An example where the two critical points are identified under the mating. The upper images show the two polynomials $F_a$ and $F_b$ in the mating of Example~\ref{ex:critptsidentified}, and the second row of images shows the lamination of the mating and the corresponding rational map, which has a double critical point (which can be seen just to the left of the centre of the image).} 
\label{f:DoubleCritExample} 
\end{figure}
 
  As the above proof suggests, it is possible that $F_a$ and $F_b$ to be (essentially) mateable even if there is a ray equivalence class containing both the free critical points $-a \in J(F_a)$ and $-b \in J(F_b)$. In this case, these simple critical points combine to become a double critical point for the resulting rational map. The author is grateful to Daniel Meyer for pointing out that this phenomenon is quite common when one considers matings involving polynomials having critical points in their Julia set.
 
 \begin{ex}\label{ex:critptsidentified}
 Consider the map $F_a$ which is the landing point of $\pray(1/36)$. This is the unique map in the limb $\Limb_{7/12}^+$. Now consider the mating of this map with $F_b$, the map which is the landing point of $\pray(11/36)$, which belongs to the limb $\Limb_0^-$ (see Figure~\ref{f:DoubleCritExample}). A simple calculation shows that $R_a(13/36)$ and $R_a(25/36)$ land on $-a$, the free critical point of $F_a$, while $R_b(23/36)$ and $R_b(35/36)$ land on $-b$, the free critical point of $F_b$. Furthermore, $R_a(1/12)$ lands on $v_a$ and $R_b(11/12)$ lands on $v_b$, so the two critical values belong to the same ray class (and the boundary of a tubular neighbourhood of this ray class generates a removable Levy cycle). However, the topological mating of these two polynomials exists, and is Thurston equivalent to the rational map whose Julia set is shown in the bottom right image of Figure~\ref{f:DoubleCritExample}. 
\end{ex}

\bibliographystyle{plain}
\bibliography{suffici}

\begin{thebibliography}{10}

\bibitem{Baranski}
Krzysztof Bara\'{n}ski.
\newblock From {N}ewton's method to exotic basins. {I}. {T}he parameter space.
\newblock {\em Fund. Math.}, 158(3):249--288, 1998.

\bibitem{RotSubsets}
Alexander Blokh, James~M. Malaugh, John~C. Mayer, Lex~G. Oversteegen, and
  Daniel~K. Parris.
\newblock Rotational subsets of the circle under {$z^d$}.
\newblock {\em Topology Appl.}, 153(10):1540--1570, 2006.

\bibitem{ACUB}
Araceli Bonifant, Xavier Buff, and John Milnor.
\newblock Antipode preserving cubic maps: the fjord theorem.
\newblock {\em Proc. Lond. Math. Soc. (3)}, 116(3):670--728, 2018.

\bibitem{CP3}
Araceli Bonifant and John Milnor.
\newblock Cubic polynomial maps with periodic critical orbit, {Part~III}:
  {Tessellations}.
\newblock In preparation, 2022.

\bibitem{DynEx}
Bryan Boyd and Suzanne Boyd.
\newblock {Dynamics Explorer}.
\newblock Computer Software.
\newblock \texttt{https://sourceforge.net/projects/detool/}.

\bibitem{DouadyHubbard:Thurston}
Adrien Douady and John~H. Hubbard.
\newblock A proof of {Thurston's} topological characterization of rational
  functions.
\newblock {\em Acta. Math.}, 171:263--297, 1993.

\bibitem{PrimerMCG}
Benson Farb and Dan Margalit.
\newblock {\em A primer on mapping class groups}, volume~49 of {\em Princeton
  Mathematical Series}.
\newblock Princeton University Press, Princeton, NJ, 2012.

\bibitem{FaughtThesis}
John~Darroch Faught.
\newblock {\em Local connectivity in a family of cubic polynomials}.
\newblock PhD thesis, Cornell University, Ithaca, NY, 1992.

\bibitem{Gantmacher}
F.~R. Gantmacher.
\newblock {\em The theory of matrices. {V}ol. 1}.
\newblock AMS Chelsea Publishing, Providence, RI, 1998.

\bibitem{Goldberg}
Lisa~R. Goldberg.
\newblock Fixed points of polynomial maps. {I}. {R}otation subsets of the
  circles.
\newblock {\em Ann. Sci. \'{E}cole Norm. Sup. (4)}, 25(6):679--685, 1992.

\bibitem{HubbardTeichvol2}
John~Hamal Hubbard.
\newblock {\em Teichm\"{u}ller theory and applications to geometry, topology,
  and dynamics. {V}ol. 2}.
\newblock Matrix Editions, Ithaca, NY, 2016.
\newblock Surface homeomorphisms and rational functions.

\bibitem{NotionsofMating}
Daniel Meyer and Carsten~L. Petersen.
\newblock On the notions of mating.
\newblock {\em Ann. Fac. Sci. Toulouse Math. (6)}, 21(5):839--876, 2012.

\bibitem{MilnorComplex}
John~W. Milnor.
\newblock {\em Dynamics in One Complex Variable}.
\newblock Princeton University Press, 3rd edition, 2006.

\bibitem{CP1}
John~W. Milnor.
\newblock Cubic polynomial maps with periodic critical orbit. {I}.
\newblock In {\em Complex dynamics}, pages 333--411. A K Peters, Wellesley, MA,
  2009.

\bibitem{Combtypes}
Carsten~L. Petersen and Saeed Zakeri.
\newblock On combinatorial types of periodic orbits of the map {$x\mapsto
  kx\pmod{\mathbb Z}$}.
\newblock {\em Adv. Math.}, 361, 2020.

\bibitem{RoeschPuzzles}
Pascale Roesch.
\newblock Puzzles de {Y}occoz pour les applications \`a allure rationnelle.
\newblock {\em Enseign. Math. (2)}, 45(1-2):133--168, 1999.

\bibitem{Roesch}
Pascale Roesch.
\newblock Hyperbolic components of polynomials with a fixed critical point of
  maximal order.
\newblock {\em Ann. Sci. \'{E}cole Norm. Sup. (4)}, 40(6):901--949, 2007.

\bibitem{CubicObs}
Thomas Sharland.
\newblock Matings of cubic polynomials with a fixed critical point, {P}art {I}:
  {T}hurston obstructions.
\newblock {\em Conform. Geom. Dyn.}, 23:205--220, 2019.

\bibitem{ShishEssMate}
Mitsuhiro Shishikura.
\newblock On a theorem of {M}. {R}ees for matings of polynomials.
\newblock In {\em The {M}andelbrot set, theme and variations}, volume 274 of
  {\em London Math. Soc. Lecture Note Ser.}, pages 289--305. Cambridge Univ.
  Press, Cambridge, 2000.

\bibitem{ShishTan}
Mitsuhiro Shishikura and {Tan~Lei}.
\newblock A family of cubic rational maps and matings of cubic polynomials.
\newblock {\em Experimental Math.}, 9:29--53, 2000.

\bibitem{Tan:quadmat}
{Tan~Lei}.
\newblock Matings of quadratic polynomials.
\newblock {\em Ergodic Th. Dyn. Sys.}, 12:589--620, 1992.

\bibitem{SaeedBook}
Saeed Zakeri.
\newblock {\em Rotation sets and complex dynamics}, volume 2214 of {\em Lecture
  Notes in Mathematics}.
\newblock Springer, 2018.

\end{thebibliography}

\end{document}